\documentclass[12pt, reqno, a4paper,oneside]{amsart}
\usepackage{graphicx, epsfig, psfrag}
\usepackage{amsfonts,amssymb,amsbsy,amsthm,amsmath}
\usepackage[T1]{fontenc}
\usepackage{bm}
\usepackage{color}
\usepackage{float}
\usepackage{hyperref}
\usepackage{mathrsfs}
\usepackage{subfigure}
\usepackage{longtable}
\usepackage{graphicx}
\usepackage{bbm}
\usepackage{mathtools}

\usepackage{ulem}
\usepackage[top=2cm,bottom=2cm,left=2cm,right=2cm]{geometry}

\usepackage{pdfsync}

\usepackage{environments}
\numberwithin{theorem}{section}
\numberwithin{equation}{section}
\renewcommand{\cases}[1]{\left\{ \begin{array}{rl} #1 \end{array} \right.}
\newcommand{\smfrac}[2]{{\textstyle \frac{#1}{#2}}}

\def\R{\mathbb{R}}
\def\C{\mathbb{C}}
\def\Z{\mathbb{Z}}
\def\N{\mathbb{N}}

\def\CC{\mathrm{C}}

\def\<{\langle}
\def\>{\rangle}

\def\wto{\rightharpoonup}

\def\b{\big}
\def\B{\Big}
\def\bg{\bigg}
\def\Bg{\Bigg}

\def\hop{\mathrm{hop}}

\def\arcsinh{\mathrm{arcsinh}}

\def\Us{\mathscr{W}}
\def\Usz{\mathscr{W}_0}
\def\Hsi{\dot{\mathscr{W}}^{1,2}}
\def\Rg{\mathcal{R}}

\def\mB{{\sf B}}

\def\mR{{\sf R}}

\def\sep{\,|\,}
\def\bsep{\,\b|\,}
\def\Bsep{\,\B|\,}

\def\D{\nabla}
\def\del{\delta}
\def\ddel{\delta^2}

\def\dt{\,{\rm d}t}

\def\ds{\,{\rm d}s}

\def\Brav{\mathcal{L}}
\def\L{\Lambda}
\def\yh{\hat{y}}
\def\alh{\hat{\alpha}}
\def\Bonds{\mathcal{B}}
\def\Cells{\mathcal{C}}
\def\Cores{\mathcal{C}^\pm}
\def\Coresp{\mathcal{C}^+}
\def\Coresm{\mathcal{C}^-}
\def\E{\mathcal{E}}
\def\eps{\epsilon}

\def\conv{{\rm conv}}

\def\supp{{\rm supp}}

\renewcommand{\d}[1]{d_{#1}}
\def\psilin{\psi_{\mathrm{lin}}}
\def\DMCP{{\bf DMCP}}
\def\DOCP{{\bf DOCP}}

\begin{document}

\title[Screw Dislocation under Anti-Plane Deformation]{Existence and Stability 
of a Screw Dislocation under Anti-Plane Deformation}

\author{T. Hudson}
\address{T. Hudson \\ Mathematical Institute \\ University of Oxford \\
  Oxford OX1 3LB \\ UK}
\email{thomas.hudson@maths.ox.ac.uk}

\author{C. Ortner}
\address{C. Ortner\\ Mathematics Institute \\ Zeeman Building \\
  University of Warwick \\ Coventry CV4 7AL \\ UK}
\email{c.ortner@warwick.ac.uk}

\date{\today}

\thanks{CO was supported by the EPSRC grant EP/H003096 ``Analysis of
  atomistic-to-continuum coupling methods''. TH was supported by the
  UK EPSRC Science and Innovation award to the Oxford Centre for
  Nonlinear PDE (EP/E035027/1).}



\begin{abstract}
  We formulate a variational model for a geometrically necessary screw
  dislocation in an anti-plane lattice model at zero
  temperature. Invariance of the energy functional under lattice
  symmetries renders the problem non-coercive. Nevertheless, by
  establishing coercivity with respect to the elastic strain and a
  concentration compactness principle, we prove existence of a global
  energy minimiser and thus demonstrate that
  dislocations are globally stable equilibria within our model.
\end{abstract}

  % We demonstrate that configurations containing dislocations may be
  % obtained as global energy minimisers of this non-coercive
  % variational problem. The non-coercivity arises from the lattice
  % symmetries.  We subsequently prove existence of solutions to this
  % problem using a concentration compactness principle and the direct
  % method of the calculus of variations.

\maketitle

\section{Introduction}
Dislocations are line defects in crystalline solids which can be
described by discontinuous displacements of a homogeneous crystal: to
obtain the simplest forms of dislocations, the lattice is sliced along
some half plane (the Volterra cut), and then deformed so that the
lattice remains almost perfect away from the edge of that half plane
(the dislocation line) \cite{HirthLothe, Volterra07}.  Since
dislocations are the principal carriers of plastic flow in crystals
\cite{Orowan34,Polanyi34,Taylor34}, they are among the most widely
studied objects of materials science. They have been investigated
analytically as points in an elastic continuum \cite{HirthLothe,
  CermLeoni05, SZ10, GLP10,DislDyn,BulatovCai}, as crystal defects
using molecular simulation techniques
\cite{BulatovCai,Sinclair:1971,ShilkrotMillerCurtin2004}, or through a
variety of intermediate models such as phase field descriptions
\cite{KCO02, RLBF03, GM06}. We refer to \cite{Foll13, HirthLothe,
  HullBacon11,BulatovCai} for introductions to these various models.

% Despite the substantial effort expended on the study of dislocations,
% many aspects of their nature, and in particular of their mathematical
% models, remain incomplete.

In the present work, we focus on the atomistic structure of
dislocations. Precisely, we shall demonstrate that they can be
understood as global minima in a variational problem. To the
best of our knowledge, our results are the first that establish the
existence of dislocations as stable equilibrium configurations of an
atomistic energy.

Our work is motivated by ongoing efforts to develop multi-scale models
of dislocations such as dislocation dynamics \cite{BulatovCai}, and
far-field coarse-graining techniques such as quasicontinuum and
related methods \cite{Miller:2008, ShilkrotMillerCurtin2004,
  Sinclair1971}. Our results contribute to a precise qualitative
understanding of the atomistic structure of dislocations, which can be
used to inform the formulation and analysis of such multi-scale
schemes \cite{2012-MATHCOMP-qce.pair,bqce.alg}.

Further, our work is inspired by a recent effort to place the theory
of dislocations on a rigorous mathematical foundation, in particular
clarifying the connections between the various models mentioned above
\cite{GPPS12,VCOA07,Ponsiglione07,GM05}. We outline only a small
fraction of the contributions here, most closely related to our own
work. We believe that, to some extent, our results help to
  overcome simplifying assumptions made in many of these works.

Possibly the most complete analysis of a static model of
dislocations is provided in \cite{GM05, GM06, CGM11}. This series of
papers studies a continuum phase field model for dislocations in a
periodic environment of pinning sites
first described in \cite{KCO02}. The authors obtain a variety of scaling
regimes, depending upon the number of the pinning sites relative to
their size in terms of the interatomic spacing.

A mathematically consistent description of dislocations in an
atomistic setting, using the language of algebraic topology is given
in \cite{ArizaOrtiz05}. The concepts and language developed therein are
used to derive models of the elastic energy of a
dislocation configuration. A rigorous discrete-to-continuum passage
within this framework is established in \cite{Ponsiglione07}
using the language of $\Gamma$-convergence.
%
%  studies a
% discrete-to-continuum limit of an energy describing dislocations in a
% discrete lattice, but considered as a function over a measure
% describing dislocation core positions. 
%
Related works analysing dislocations and other similar defects in discrete
systems are \cite{AlicCic09,ACP11}.

The works cited above concerning discrete and semidiscrete models of
screw dislocations are primarily concerned with asymptotics of the elastic
stored energy, given a number of prescribed dislocation cores. The creation
or destruction of additional cores, for example via the introduction of
dipoles, is either forbidden or explicitly tracked in the energy functional
through a term accounting for a {\em positive} core energy.

In contrast, our model allows for the creation and destruction of
dislocation dipoles without any such penalty, accounting
only for the stored elastic energy. More precisely, we consider a static
atomistic model for screw dislocations in a similar vein to
\cite{Ponsiglione07}, and show that {\em unconstrained} stable
equilibrium configurations containing dislocations exist. Although we
do not pursue this in the present work, we observe that the analytic
properties of the equilibria we obtain should allow for a natural
extension of the results in \cite{Ponsiglione07} to our unconstrained
model.

\subsection{Outline}
We consider anti-plane displacements of a BCC crystal, in the
direction of a screw dislocation line (that is, parallel to the
Burgers vector). For two displacements, $y, \tilde{y}$ we consider the
energy difference
\begin{equation*}
  E(y;\tilde{y}):=\sum_{b\in\Bonds} \b[\psi(D\tilde{y}_b)-\psi(Dy_b)\b],
\end{equation*}
where $\Bonds$ is a set of pairs of interacting (lines of) atoms, $Dy_b$
denotes a finite difference, and $\psi$ is a potential describing this
interaction. The potential $\psi$ is $1$-periodic, where $1$ is the
atomic spacing, mimicking the behaviour of realistic pair interaction
potentials. In particular, if
$y(\xi) - \tilde{y}(\xi) \in \Z$ for all lattice sites $\xi$, then
$E(y; \tilde{y}) = 0$. This invariance of the energy is the primary
source of analytical difficulties.

We call a deformation $y$ which minimises $E(y+u;y)$ amongst all
finite energy perturbations $u$ a {\em globally stable equilibrium}.

Building on \cite{ArizaOrtiz05} and the
differential displacement maps first employed in \cite{VPB70},
in \S \ref{sec:bonds} we
present a method by which we can identify dislocation cores and assign
them a corresponding Burgers vector.
%  This then leads to our
% definition of net Burgers vector in \S \ref{sec:burgers_vec}.

Our main result, Theorem \ref{th:ex_nec_disl}, states that there
exists a globally stable equilibrium containing 
one geometrically necessary dislocation core.

The proof of this result is developed throughout the remainder of the
paper. In \S \ref{sec:Ediff} we show that $\E(u) := E(\yh+u; \yh)$,
where $\yh$ is the linearised elasticity displacement field, is
well-defined on a discrete $H^1$-function space describing finite
energy states. In Theorem \ref{th:minim_E} we reformulate and refine
Theorem \ref{th:ex_nec_disl} as a variational problem, by stating the
existence of a global minimiser of~$\E$ in that space.
\S \ref{sec:mainproof} is then devoted to the proof of Theorem
\ref{th:minim_E}.

Further, in \S \ref{sec:main_result} we discuss briefly how to expoit
the global stability result of Theorem \ref{th:minim_E} to construct
{\em locally stable} configurations containing finitely many
dislocation cores or configurations with dislocations in a domain with
boundary.

\section{Geometric and Topological Preliminaries}

% In this section we lay out the terminology and notation which we will
% use to analyse lattices and dislocations, which are geometric lattice
% defects.

\subsection{Anti-plane displacements of the BCC lattice}
\label{subsec:lattices}

Although our analysis can be applied in other situations, it
will be notationally convenient and physically relevant to restrict
our attention to the body-centred cubic (BCC) lattice, which may be
defined by
\begin{displaymath}
  \Brav:=\mB\Z^3 \qquad \text{where} \qquad \mB := s \left[ \begin{array}{ccc}
      \sqrt{8/9} & \sqrt{2/9} & 0 \\
      0 & \sqrt{2/3} & 0 \\
      1/3 & -1/3 & 1
    \end{array} \right],
\end{displaymath}
and $s > 0$ is a scaling factor that we leave undefined for now.

To define anti-plane displacements, we fix the lattice vector $\nu :=
[0, 0, s]$ and define the projection
\begin{align*}
  \Pi_\nu:=\mathsf{I}-\smfrac{\nu}{|\nu|} \otimes \smfrac{\nu}{|\nu|},
\end{align*}
which `flattens' the Bravais lattice $\Brav$ onto the lattice plane
with normal $\nu$; see Figure \ref{fig:bravlatt}. It is
straightforward to check that the set $\Pi_\nu (\Brav)$ is a
two-dimensional triangular lattice embedded in $\R^3$, with lattice
constant $s\sqrt{8/9}$; we will choose $s=\sqrt{9/8}$ so that the
planar lattice constant is $1$. We further shift the origin so that
the projected set may be identified with
\begin{align*}
  \L:= \b(\smfrac12,\smfrac{\sqrt{3}}6\b)^T + [\,a_1, a_2 ] \cdot \Z^2, \quad
  \text{where } a_1 = (1,0)^T \text{ and } a_2 = \b(\smfrac12 ,
  \smfrac{\sqrt{3}}{2} \b)^T.
\end{align*}

An {\em anti-plane displacement} in the direction $\nu$ (or, simply,
{\em displacement}), is a map $y : \L \to \R$. The set of all
displacements is denoted by $\Us$.  A displacement $y$ gives rise to a
lattice {\em deformation} $Y : \Brav \to \R^3$,
\begin{displaymath}
  Y(\eta) := \eta + y\b( (\smfrac12, \smfrac{\sqrt{3}}{6})^T  +  \Pi_\nu \eta \b)
  \nu,  \qquad \eta \in \Brav.
\end{displaymath}

Let $y, \tilde{y}$ be displacements and $Y, \tilde{Y}$ the associated
deformations. We say that $y, \tilde{y}$ are equivalent if $Y(\Brav) =
\tilde{Y}(\Brav)$ (i.e., they describe the same atomistic
configurations). It is easy to see that $y, \tilde{y}$ are equivalent
if and only if $(y - \tilde y)(\L) \subseteq \Z$.

% With $\L$ as defined in Section \ref{subsec:lattices}, we define $\Us$ to be the
% set of all real-valued functions on $\L$, and $\Usz$ to be those functions with
% compact support:
% \begin{align*}
%   \Us:=\b\{y\bsep y:\L\to\R\b\},\qquad
%     \Usz:=\b\{y\in\Us\bsep \supp(y)\text{ is compact}\b\}.
% \end{align*}
% We note that $y$, the anti--plane component of $Y$ as defined in \eqref{eq:Ydef}
% lies in $\Us$.

\subsection{Bonds and bond lengths}
\label{sec:bonds}
Each $\xi \in \L$ has six nearest neighbours, $\xi+a_i$, $i = 1,
\dots, 6$, where $a_1 = (1, 0)^T$ and $a_i = \mR_6^{i-1} a_1$ where
$\mR_6$ denotes a rotation through angle $\pi/3$.  At a point $\xi\in\L$,
we define the set of outward-pointing nearest neighbour bonds
\begin{equation*}
  \Rg_\xi:=\b\{(\xi, \xi+a_i) \bsep i = 1, \dots, 6 \b\}=\b\{(\xi,\eta)\bsep \eta\in\L,
    |\xi-\eta|=1\b\},
\end{equation*}
and furthermore define the set of all bonds to be the union
\begin{align*}
  \Bonds:= \bigcup_{\xi\in\L}\Rg_\xi= \b\{(\xi, \xi+a_i) \bsep
    \xi \in \L, i = 1, \dots, 6 \b\} =\b\{(\xi, \eta)\in\L^2 \bsep
    |\xi-\eta|= 1\b\}.
\end{align*}
For any bond $b=(\xi, \xi+a_i)$, we denote the reverse bond by $-b :=
(\xi+a_i, \xi)$. 

For $b = (\xi,\eta) \in \Bonds$ we define the difference operator 
\begin{align*}
  Dy_b := y(\eta)-y(\xi).
\end{align*}
Moreover, we set $Dy := (Dy_b)_{b \in \Bonds}$. We also note that
$Dy_{-b}=-Dy_b$.

With this notation, we can now define two important discrete function
spaces: fixing a reference lattice point $\xi_0 =(0,\frac{\sqrt{3}}{3})^T$,
\begin{align*}
  \Usz &:= \b\{ v \in \Us \bsep v(\xi_0) = 0 \text{ and } {\rm supp}(Dv) \text{ is bounded}
  \b\}, \qquad \text{and}  \\
  \Hsi &:= \b\{ v \in \Us \bsep v(\xi_0) = 0 \text{ and } Dv \in
  \ell^2(\Bonds) \b\}.
\end{align*}
It is shown in \cite[Prop. 9]{OrtSha:interp:2012} that $\Hsi$ is a Hilbert
space and $\Usz \subset \Hsi$ is dense.
% Imposing the condition
% $v(0) = 0$ does not alter the physical configuration, but ensures that
% $\| D \cdot \|_{\ell^2}$ is a norm.

% Due to the equivalence of displacements discussed in
% \S\ref{subsec:lattices}, modifying $y$ by an integer at any point
% gives rise to an identical associated deformation $Y$.

We now introduce a crucial concept required to define the notation of
dislocation. We denote the set of {\em bond length 1-forms} by
\begin{equation}
  \label{eq:defn:[Dy]}
  [Dy] := \b\{\alpha : \Bonds \to [-1/2, 1/2] \bsep  \alpha_{-b} =
  \alpha_b \text{ and } Dy_b - \alpha_b \in \Z \text{ for all } b \in
  \Bonds \b\}.
\end{equation}
We note that, if $Dy_b \not\in \smfrac12 + \Z$ for all $b \in \Bonds$,
then $\alpha \in [Dy]$ is unique, but in general there is ambiguity
in the definition of $\alpha$. This non-uniqueness is an issue which
we will return to in \S\ref{sec:burgers_vec}.

The motivation behind this definition is that $\alpha_b$ defines the
`shortest bond length' between the two lines of nuclei represented by
the 2D lattice sites $\xi, \xi'$, where $b = (\xi,\xi')$, in that
\begin{displaymath}
  \min_{\substack{\eta, \eta' \in \Brav \\
      \Pi_\nu \eta=\xi, \Pi_\nu \eta' = \xi'}} \b|y(\eta) - y(\eta')\b| =   \sqrt{1 + \alpha_b^2};
\end{displaymath}
see also Figure \ref{fig:3Dsideview}.

\begin{figure}
  \subfigure[Part of $\Brav$, showing the BCC unit cell,
    $\nu$ and the plane perpendicular to it.]{
    \includegraphics[width=0.45\textwidth]{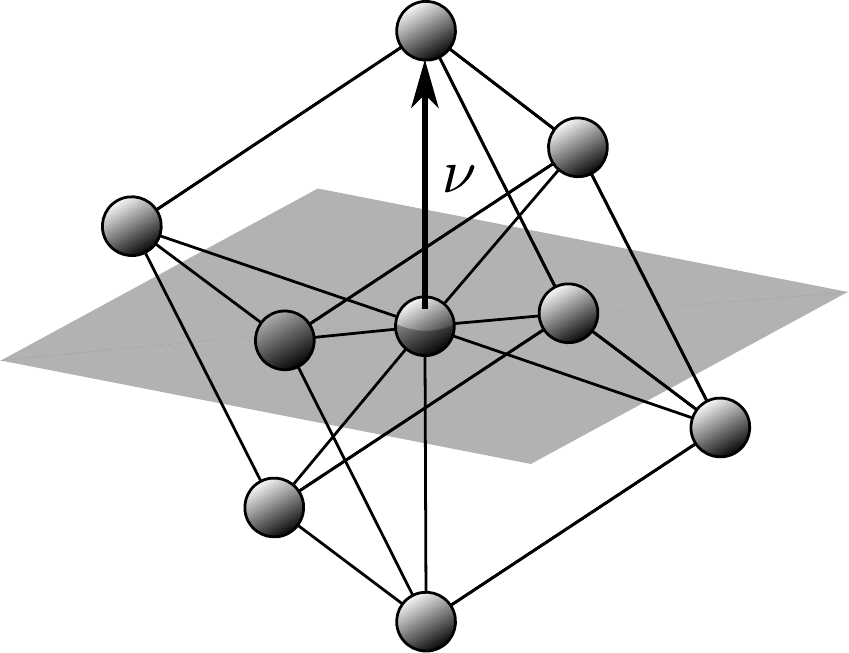}
    \label{fig:bravlatt}
  } 
  \hspace{0.03\textwidth}
  \subfigure[An illustration of the definition of $\alpha_b$ and its
  relationship to the shortest distance between atoms.]{ 
    \includegraphics[width=0.45\textwidth]{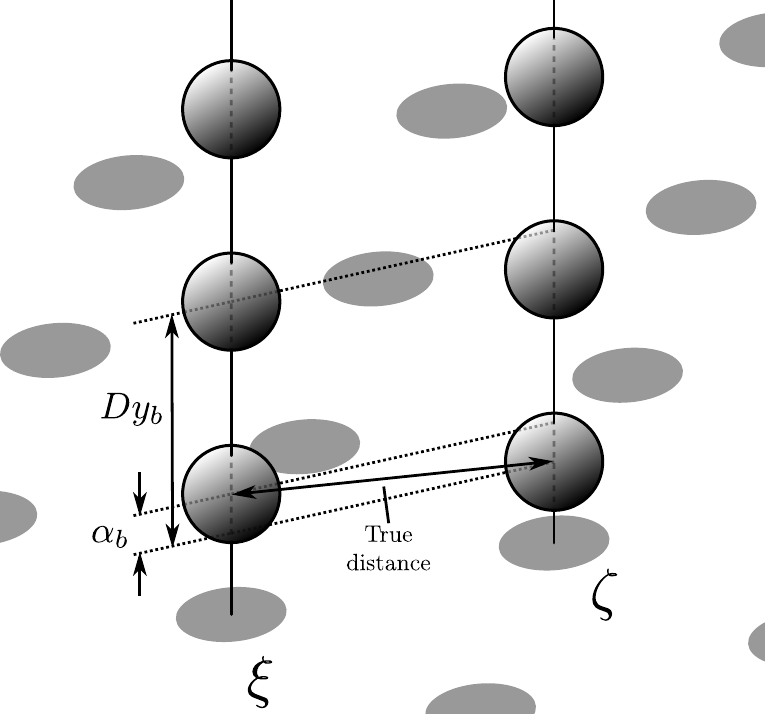}
    \label{fig:3Dsideview}
  }
  \caption{Illustrations of the lattice geometry.}
\end{figure}

The importance of the concept of bond length stems from the fact that,
due to the invariance of the lattice under adding integer shifts to a
displacement, the energy of the lattice can only depend on $\alpha_b$,
but not on $Dy_b$ directly.

\subsection{The lattice complex}
\label{sec:latt.complex}
In this section, we review some terminology of discrete algebraic topology
which is convenient for our analysis. We follow the language described
in \cite{ArizaOrtiz05}, where more details and applications to the
study of dislocations can be found.

Repeating the definitions of $\L, \Bonds$, we define a lattice complex
as in \cite[\S 2.3.3]{ArizaOrtiz05}, with
\begin{align*}
  \L &:= \b\{\xi\in\R^2\bsep\xi\in\Pi_\nu\Brav + (\smfrac12,
  \smfrac{\sqrt{3}}{6})^T \b\},\\
  \Bonds &:= \b\{(\xi,\zeta)\in\L^2\bsep |\xi-\zeta|=1\b\}, \quad
  \text{and} \\
  \Cells &:= \b\{(\xi,\zeta,\eta)\in\L^3\bsep (\xi,\zeta),
  (\zeta,\eta),(\eta,\xi)\in\Bonds\},
\end{align*}
denoting, respectively, the sets of 0-cells, 1-cells and
2-cells of the lattice complex respectively (see Figure
\ref{fig:cells} for an illustration). From now on, we will not
explicitly use the terms $p$-cell, $p$-chain and $p$-cochain as
defined in \cite[\S 2.2]{ArizaOrtiz05}, preferring instead the more
evocative terminology `lattice points' for elements of $\L$, `bonds'
for elements of $\Bonds$, and `cells' for elements of $\Cells$. We
note the additive structure that may be defined on these objects, and
write $a\in A$ to mean that $a$ is an elementary $p$-cell contained in
the $p$-chain $A$. We also frequently use the boundary operator
$\partial$, which maps $p$-chains to their boundaries, assigning
orientations in the usual way.

We then follow \cite[\S 3]{ArizaOrtiz05} in defining $p$-forms
and integration on the lattice, writing
\begin{equation*}
  \int_{U}F:=\sum_{e\in U} F(e),
\end{equation*}
where $U$ is a $p$-chain, $e$ are $p$-cells, and $F$ is a $p$-form (i.e. a
real-valued function on $p$-cells). We note that this definition is linear in
$F$ and $U$, in the sense that
\begin{align*}
  \int_{U+V}\lambda F+G &=\lambda \sum_{e\in U}F(e)+\lambda \sum_{e\in V}F(e)+\sum_{e\in U}G(e)+\sum_{e\in V}G(e),\\
  &=\lambda \int_UF+\int_UG+\lambda \int_VF+\int_VG,
\end{align*}
for any $\lambda\in\R$, $p$-chains $U$ and $V$, and $p$-forms $F$ and
$G$. We remark here that `bond length 1-forms' $\alpha$ as defined in
\S\ref{sec:bonds} are true 1-forms in the sense defined in
\cite[\S3.1]{ArizaOrtiz05}.

\begin{figure}
  \includegraphics[height=6cm]{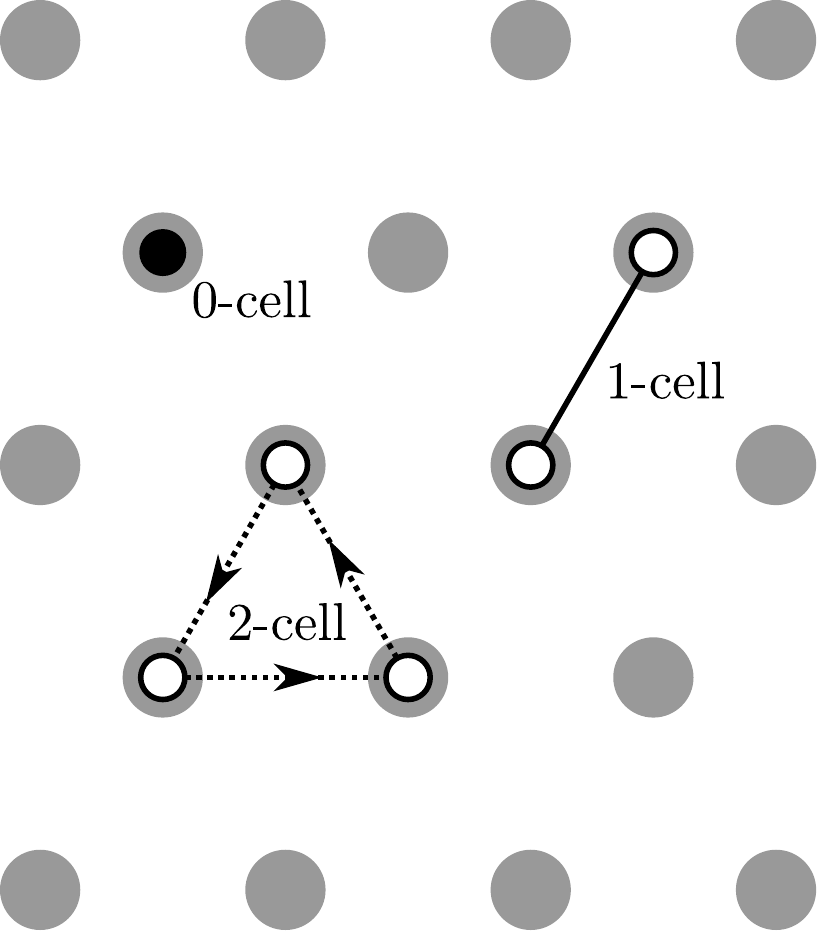}
  \caption{An illustration of 0-, 1- and 2-cells in the triangular
    lattice. The arrows show the boundary of the positively-oriented 2-cell.
    Note that orientation only makes sense for 1- and 2-cells.}
  \label{fig:cells}
\end{figure}

Finally, we define paths in the natural way as 1-chains
\begin{equation}
  \Gamma:=\sum_{k=1}^L(\xi_k,\xi_{k+1}),
  \label{eq:pathdef}
\end{equation}
where $(\xi_k,\xi_{k+1})\in\Bonds$ for each $k$, and we denote the
length of a path $\Gamma$ by $|\Gamma|:=L$.

\subsection{Measures of Lattice Distance}
\label{sec:latt.dist}
Since we will make use of more than simply the algebraic structure that a
lattice complex entails, we will occasionally abuse the notation given above by
identifying bonds and cells with their closed convex hulls; that is, we write
\begin{align*}
  x\in b&=(\xi,\zeta) \qquad\text{to mean}\quad
  x\in\conv\{\xi,\zeta\}, \qquad \text{and} \\
  x\in C&=(\xi,\zeta,\eta) \quad\text{to mean}\quad x\in\conv\{\xi,\zeta,\eta\},
\end{align*}
where $\conv(\Omega)$ denotes the closed convex hull of a set
$\Omega\subset\R^2$; it will be clear from the context whether we are
referring to spatial points or subchains.  Since we frequently
  refer to them, we define $x^C$ to be the barycentre of a
  cell $C$, and $C_0$ the cell for which $x^{C_0}=0$.

Using this form of the notation, we define the distance from each kind of cell to
the origin as
\begin{align*}
  \d{\xi}&:=|\xi|,\\
  \d{b}&:=\inf_{x\in b}|x|,\\
  \d{C}&:=\inf_{x\in C}|x|,
\end{align*}
which corresponds to the usual notion of distance between sets in
Euclidean space.

The second notion of distance we will use is the graph theoretic
notion.  Since $\L$ can be identified with a planar graph with edges
$b\in\Bonds$, we can further identify cells with nodes in the dual
graph, and bonds as edges in this graph (see
\cite[\S4.6]{Diestel10}). This allows us to define the {\em hopping
  distance}, $\hop_2(C,C')$ as the length of the shortest path in the
dual graph between the cells $C,C'\in\Cells$, as in
\cite[\S1.3]{Diestel10}. We note that since the dual graph is
connected, this distance is always finite, and we have the following
`triangle inequality' for any dual lattice points $A$, $B$ and $C$:
\begin{equation}
  \hop_2(A,C)\leq \hop_2(A,B)+\hop_2(B,C).\label{eq:hop_tri_ineq}
\end{equation}

\subsection{The Burgers vector}
\label{sec:burgers_vec}
We now define the notion of Burgers vector we use in our model, which
is a fundamental geometric concept describing the nature of a
dislocation \cite{HirthLothe}.

We call a path $\Gamma=\sum_{k=1}^L(\xi_k,\xi_{k+1})$ as defined in
\eqref{eq:pathdef} a {\em Burgers loop} (or, simply, {\em loop}) if
$\xi_1=\xi_{L+1}$ (or equivalently $\partial_1\Gamma=0$; this implies
$\Gamma=\partial_2 A$ for some sum of cells $A$, since the
lattice complex is perfect \cite[\S 2.2.1, Axiom (A4)]{ArizaOrtiz05}).
If $\Gamma$ is a loop, $y\in\Us,$ and $\alpha \in [Dy]$ an associated
bondlength 1-form, then
\begin{align*}
  \int_\Gamma \alpha = \sum_{b\in\Gamma} Dy_b+ \sum_{b
    \in\Gamma} (\alpha_b - Dy_b) = 0 + N \in \Z.
\end{align*}
We call the integer $N$ the Burgers vector of the bond length 1-form
$\alpha$ around the loop $\Gamma$.

% Following from the bilinearity of \cco{***the of
%   the duality product***}, if $\Gamma_i$, $i=1,\ldots,K,$ are loops
% and
% \begin{align*}
%   \Gamma := \sum_{i=1}^K \Gamma_i,
% \end{align*}
% then for any deformation $y\in\Us$,
% \begin{align*}
%   \int_\Gamma \alpha = \sum_{i=1}^K \int_{\Gamma_i}\alpha.
% \end{align*}
% We refer to this relation as the conservation of Burger's vector, in keeping
% with the language used in this area of study (see for example `Equivalent
% Burgers Circuits', pp.24-25 in \cite{HirthLothe}).

\begin{definition}[Dislocation Core]
  \label{def:core}
  A dislocation core of a bond length 1-form $\alpha$ is a positively oriented
  2-cell $C$ such that $\int_{\partial C} \alpha \neq 0$.
\end{definition}
\medskip

We will refer to cores as being `contained in' $\alpha$. For future reference, we
remark that
\begin{equation}
  \label{eq:bvec_cell_leq_1}
  \bg|\int_{\partial C}\alpha\,\bg| \leq \frac{|\partial C|}{2} = \frac{3}{2},
\end{equation}
that is, the Burgers vector around a single 2-cell can only be $-1, 0$ or
$1$, and hence we define the sets of dislocation cores
\begin{align*}
  \Coresp[\alpha]&:=\bg\{ C\in\Cells\Bsep C\text{ positively oriented, }\int_{\partial C}\alpha=+1\bg\},\\
  \Coresm[\alpha]&:=\bg\{ C\in\Cells\Bsep C\text{ positively oriented, }\int_{\partial C}\alpha=-1\bg\},\\
  \Cores[\alpha]&:=\Coresp[\alpha]\cup\Coresm[\alpha].
\end{align*}

\begin{remark}
\label{rem:alpha_ambig}
It is interesting to note that if $\alpha, \alpha' \in [Dy]$, then
they need not have the same number of cores; see Figure \ref{fig:bnd.len.ex} for
an illustration of this fact.

The only point at which this ambiguity is an issue is if $\alpha$
has $C,C'\in\Cores[\alpha]$ which are adjacent. In that case, it may be
checked that the $b\in\partial C$ such that $-b\in\partial C'$ must
have $\alpha_b\in\{-1/2,0,1/2\}$. In the case where $\alpha_b=\pm1/2$,
redefining $\alpha_b=\mp1/2$ removes these cores, and $\alpha$ remains
a bond length 1-form in $[Dy]$, so we will always assume that
minimising sequences have $\alpha_b=0$ for any bond $b$ shared by 2
adjacent cores.
\end{remark}
\medskip

\begin{figure}
  \includegraphics[height=2.5cm]{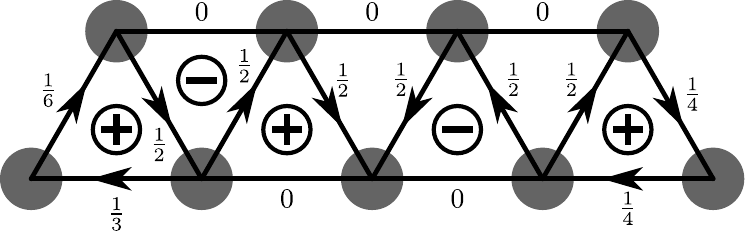}
  
  \medskip

  \includegraphics[height=2.5cm]{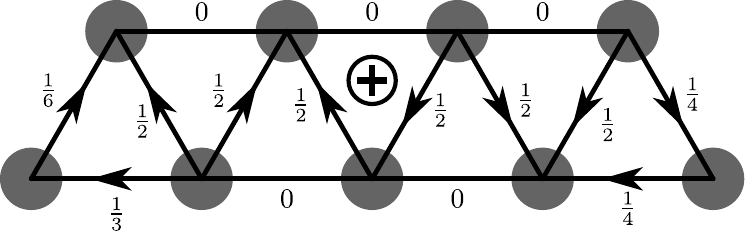}
  \caption{Two examples of bond length 1-forms corresponding to the
  same deformation. The numbers are the value of the 1-form on the
  relevant bond, and arrows indicate the bond direction in which
  it is positive. Note the number and positions of the dislocation
  cores present change, but the sum of the Burgers vectors does not.}
  \label{fig:bnd.len.ex}
\end{figure}

The {\em net Burgers vector} is obtained by summing the signs of the
cores, or equivalently, by computing the Burgers vector on a
sufficiently large loop enclosing all cores. Since $\alpha \in [Dy]$
is not necessarily unique for a given $y$, we ensure that such a concept
can be defined unambiguously.

For our purposes it will be enough to consider displacements with some
prescribed far-field behaviour.

\begin{proposition}
  Let $y \in \Us$ and $\alpha \in [Dy]$ such that $\alpha_b \to 0$ as
  $\d{b}\to\infty$. Then, for any $\alpha' \in [Dy]$,
  $\Cores[\alpha']$ is finite and
  \begin{equation}
    \label{eq:B[y]_ind_of_alpha}
    \sum_{C\in\Cores[\alpha']}\int_{\partial C}\alpha' =
    \sum_{C\in\Cores[\alpha]}\int_{\partial C}\alpha =
    \int_{\Gamma} \alpha,
  \end{equation}
  where $\Gamma$ is any loop that encloses all cores
  in $\alpha$.
\end{proposition}
\begin{proof}
  If $\alpha_b\to 0$ as $\d{b}\to\infty$ then
  \begin{align*}
    \int_{\partial C} \alpha \to 0 \qquad \text{ as } \d{C}
    \to \infty.
  \end{align*}
  Since $\int_{\partial C} \alpha \in \Z$ it follows that
  $\int_{\partial C} \alpha = 0$ for $\d{C}$ sufficiently large, and
  hence the number of dislocation cores present in $\alpha$ is finite.

  Moreover, since $\alpha_b \in (-1/2, 1/2)$ for $\d{b}$ sufficiently
  large, it follows that $\alpha_b = \alpha_b'$ for all $\alpha' \in
  [Dy]$ and $\d{b}$ sufficiently large. In particular,
  $\Cores[\alpha']$ is also finite.

  To prove \eqref{eq:B[y]_ind_of_alpha}, let $\Gamma$ be a loop
   that encloses all the cores in $\alpha$ for
    which $\Gamma = \partial A$. Then
  \begin{displaymath}
    \int_\Gamma \alpha = \sum_{C \in A}
    \int_{\partial C} \alpha = \sum_{C \in \Cores} \int_{\partial C} \alpha.
  \end{displaymath}
  Taking $\Gamma$ such that $Dy_b = \alpha_b = \alpha_b'$ for all $b
  \in \Gamma$ we obtain the first identity in
  \eqref{eq:B[y]_ind_of_alpha} as well.
  %
  % \cco{I think it is enough to say 
  %
  %   ``Taking $\Gamma$ such that $Dy_b = \alpha_b = \alpha_b'$ for all
  %   $b \in \Gamma$ we obtain the first identity in
  %   \eqref{eq:B[y]_ind_of_alpha} as well.''
  % }
  %
  % Next, note that by definition $\alpha_b-\alpha'_b\in\{-1,0,1\}$ for
  % each bond $b$, and is non-zero only for a finite number of bonds.
  % We can therefore decompose it as a finite sum of
  % \begin{equation*}
  %   \mathbbm{1}^{b'}_b:=\cases{
  %     \pm 1 & b=\pm b',\\
  %     0     & \text{otherwise.}
  %   }
  % \end{equation*}
  % Noting that $\mathbbm{1}^{b'}$ satisfies
  % \begin{equation*}
  %   \int_\Gamma \mathbbm{1}^{b'}=0,
  % \end{equation*}
  % where $\Gamma$ encircles all the cores in $\alpha$ and in $\alpha'$,
  % the first identity in \eqref{eq:B[y]_ind_of_alpha} follows as well.
\end{proof}

We can now formally define the {\em net Burgers vector}.

\begin{definition}[Net Burgers Vector]
  \label{def:netB}
  Let $y \in \Us$ such that $\alpha_b \to 0$ as $\d{b}\to\infty$ for
  some $\alpha \in [Dy]$. Then we define the {\em net Burgers vector}
  of $y$ to be
  \begin{displaymath}
    B[y] := \sum_{C\in\Cores[\alpha]} \int_{\partial C} \alpha,
  \end{displaymath}
  for an arbitrary $\alpha \in [Dy]$.
\end{definition}

\medskip

The quantity $B[y]$ can be experimentally observed from outside the
system, by determining the strain at `infinity'. For example, if $B[y]
= 1$, then this tells the observer that there must be at least one
dislocation in the system, but nothing about the total number.

\section{Main Result}
\label{sec:main_result}
In this section we present the main result of the paper with accompanying
assumptions.

\subsection{Energy difference functional}
\label{sec:Ediff_defn}
Before we can state the main result, we introduce another key concept
that we employ in its formulation and proof: the energy difference
functional. We assume that lattice sites (corresponding to lines of
atoms in the BCC crystal) interact via a nearest-neighbour pair
potential $\psi \in C(\R) \cap C^4(\R\setminus (\Z+1/2))$, which
satisfies the following properties: \\[1mm]
\qquad
\begin{minipage}{10cm}
  \begin{itemize}
  % \item[($\psi$0)] $\psi$ is four times differentiable at $0$;
  \item[($\psi$1)] $\psi$ is 1-periodic;
  \item[($\psi$2)] $\psi$ and $\psi(\smfrac12 + \cdot)$ are even;
  \item[($\psi$3)] $\psi(r) = 0$ if and only if $r \in \Z$;
  \item[($\psi$4)] $\psi''(0) = \mu > 0$.
  \item[($\psi$5)] $\psi(x)\geq \smfrac12\psi''(0)\,x^2$ for all
    $x\in[-\smfrac12,\smfrac12]$.
  \end{itemize}
\end{minipage}

\begin{remark}
  Assumptions $(\psi1)-(\psi4)$ are very general, and are natural in
  the physical context: $(\psi1)$ and $(\psi2)$ encode lattice
  symmetries, while $(\psi3)$ and $(\psi4)$ state that the system has
  a stable crystalline ground state.

  The only ``technical'' assumption is $(\psi5)$. The reason for this
  assumption will become apparent in \S\ref{sec:lower_bound}, where we
  use it to establish an a priori bound on the number of dislocation
  dipoles in finite energy configurations. We believe that $(\psi5)$
  can be replaced with weaker variants, but cannot be removed
  altogether.
  
  We remark that the requirement that $\psi\in\CC^4(\R\setminus (\Z+1/2))$
  can be relaxed further by modifying the proofs we give below, but since
  this adds little at the expense of readability, we omit such arguments
  here.
  The prototypical example of a potential satisfying $(\psi1)-(\psi5)$
  is $\psi(r) = \psi^{\rm lin}(r) := \frac12 {\rm dist}(r,
  \Z)^2$. 
\end{remark}

\medskip
For two displacements $y, \tilde{y} \in \Us$ we define the energy difference
functional, formally for the moment, as
\begin{displaymath}
  E(y; \tilde{y}) := \sum_{b \in \Bonds} \b[ \psi(Dy_b) -
  \psi(D\tilde{y}_b) \b].
\end{displaymath}
For example, if $y - \tilde{y} \in \Usz$, then $E(y; \tilde{y})$ is
clearly well-defined since the sum is effectively finite. For
arbitrary displacements $y, \tilde{y}$, $E(y; \tilde{y})$ need not be
well-defined. However, we will show in \S\ref{sec:ext_Ediff} that $E$
can, under certain conditions, be extended by continuity to relative
displacements $y - \tilde{y} \in \Hsi$.

Using the terminology of energy differences, we can define what we
mean by a stable equilibrium displacement. Intuitively, the definition
entails that finite energy perturbations cannot lower the energy.

\begin{definition}[Stable Equilibrium]
  \label{def:defn_stable_equilib}
  A displacement $y\in\Us$ is a {locally stable equilibrium} if there exists
  $\epsilon > 0$ such that $E(y + u; y) \geq 0$ for all $u \in \Usz$
  with $\| D u \|_2 \leq \epsilon$.

  A displacement $y\in\Us$ is a {globally stable equilibrium} if $E(y
  + u; y) \geq 0$ for all $u \in \Usz$.
\end{definition}

\subsection{Statement of the main result}
Recalling Definitions \ref{def:netB} and \ref{def:defn_stable_equilib}
the existence of a screw dislocation can be formulated as follows.

\begin{theorem}[Existence of a geometrically necessary dislocation]
  \label{th:ex_nec_disl}
  There exists a globally stable equilibrium displacement $y\in\Us$ with
  net Burgers vector $B[y] = 1$.
\end{theorem}

\medskip

The notion of global stability described in Definition
\ref{def:defn_stable_equilib} is equivalent to the statement that a
displacement is stable if any finite energy perturbation increases the
energy of the system.  Describing a dislocation configuration as the
minimiser of an energy difference functional gives us access to the
Direct Method of the Calculus of Variations.

We refer to this result as the existence of a `geometrically
necessary' dislocation since we do not prescribe the absolute number
of dislocation cores, but only the net Burgers
  vector.

\begin{proof}[Outline of the proof of Theorem \ref{th:ex_nec_disl}]
  \hfill
  % The proof proceeds in the following manner:
  \begin{enumerate}
  \item We define a reference configuration $\yh(\xi) = \frac{1}{2\pi}
    \arctan(\frac{\xi_2}{\xi_1})$ (the continuum linear elasticity
    solution for a dislocation), with the aim to minimise the energy
    difference functional $\E(u) := E(\yh+u; \yh)$ over a suitable
    class of functions $u$. 

    We show that this functional, initially defined over $\Usz$, can
    be continuously extended to a functional over $\Hsi$.

  \item In order to use the Direct Method to establish the existence
    of a minimiser to $\E$, the crucial step is to obtain a global
    lower bound on the energy. This is the main step in the proof, and
    requires careful geometric estimates based on the number of
    dislocation cores and the distance between them. We shall prove
    that $\E(u) \gtrsim \|\beta\|_2^2 - 1$, where $\beta$ can be
    thought of as belonging to $[Du]$ (however, see
    \eqref{eq:defn_beta} for the precise definition).

  \item This lower bound guarantees in particular that the number of
    dislocation cores is bounded along a minimising sequence as well
    as weak compactness of a minimising sequence $u^n$.

  \item The final step is to ensure that $\lim u^n$ has non-zero
    net-Burgers vector. This need not be the case since weak
    convergence of $D u^n$ allows for energy to be translated to
    infinity. In our present context it is possible, by introducing a
    dislocation dipole, to effectively translate the geometrically
    necessary core to infinity, and thus obtain a limiting
    displacement with zero net Burgers vector. We shift the minimising
    sequence and employ a concentration compactness argument to
    prevent this.
  \end{enumerate}
\end{proof}

\subsection{Locally stable equilibria}
Theorem \ref{th:ex_nec_disl} establishes the existence of a
configuration $y = \yh + u$, which is a {\em globally stable}
equilibrium configuration for a single screw dislocation in an
infinite lattice. From this starting point, it is possible to
construct more general {\em locally stable} equilibrium
configurations. The idea is (1) to superimpose copies of $y$ and define
\begin{displaymath}
  \tilde{z}(\xi) := \sum_{j = 1}^J s_j y\b( \xi - x^{C_j} \b),
\end{displaymath}
where $C_j \in \Cells$ are cores in $\tilde{z}$ and $s_j \in \{\pm
1\}$ the Burgers' vectors of these cores; (2) to show that $\tilde{z}$
is an approximate equilibrium when the cores $C_j$ are sufficiently
far from one another; and (3) to apply the inverse function theorem to
establish the existence of an equilibrium $z$ close to $\tilde{z}$.

Here, we only state two results that we obtain by this strategy, but
refer for their proofs to~\cite{HudsOrt:disl_ift}, where we present
them in a more general context.

In step (3) of the strategy outlined above we require a discrete
ellipticity condition \eqref{eq:ellipticity}, which can be established
rigorously, for example, for a piecewise quadratic potential.

\begin{lemma}[Discrete ellipticity]
  Let $\psi(r) := \psilin(r) := \smfrac{\lambda}{2}{\rm dist}(r,
  \Z)^2$ and let $y = \yh + u$, $u \in \Hsi$, be a locally stable
  equilibrium configuration. Then $Dy_b \in \R \setminus (\frac12
  +\Z)$ for all $b \in \Bonds$, and hence
  \begin{equation}
    \label{eq:ellipticity}
    % \< H(y) v, v \> := 
    \sum_{b \in \Bonds} \psi''(Dy_b) Dv_b^2
    \geq \lambda \| D v \|_2^2 \qquad \forall v \in
    \Hsi.
  \end{equation}
\end{lemma}
\begin{proof}
  Suppose there exists a bond $b = (\xi,\zeta)$ such that $Dy_b \in
  \frac12 + \Z$. Let $z_t(\eta) := y + t \delta_{\eta,\xi}$, then a
  direct calculation shows that $E(z_t; y) < 0$ for some sufficiently
  small $t$ (either positive or negative).
\end{proof}

The two results we state in the following admit general $\psi$, but
require \eqref{eq:ellipticity} as an assumption:

\medskip

\begin{description}
\item[{\bf (STAB)}] There exists a locally stable equilibrium $y = \yh
  + u$, $u \in \Hsi$, satisfying the ellipticity condition
  \eqref{eq:ellipticity}. Moreover, let $A$ be a finite union of cells
  such that $\Cores[\alpha] \subset A$, for any $\alpha \in [Dy]$.
\end{description}

\medskip

Our first local stability result states that any configuration of
dislocations is stable provided that the cores are sufficiently
separated. In particular, it shows that there exist stable
configurations with arbitrary net Burgers.

\begin{corollary}[Finitely many cores]
  \label{th:finitely_many}
  Suppose that {\bf (STAB)} holds.
  
  Let $C_j \in \Cells, j = 1, \dots J, J \in \N$, be a finite
  collection of cells, and let $s_j \in \{\pm 1\}$. There is a minimal
  separation distance $L_0 > 0$ such that, if $\min_{i \neq j}
  |x^{C_j} - x^{C_i}| \geq L_0$, then there exists a locally stable
  configuration $z \in \Us$ such that, for any $\alpha \in [Dz]$,
  \begin{align*}
    & \Cores[\alpha] \subset \bigcup_{j = 1}^J \b( x^{C_j} + A \b) \qquad
    \text{and} \\
    & \int_{\partial (x^{C_j} + A)} \alpha = s_j, \qquad \text{for } j
    = 1, \dots, J.
  \end{align*}
  In particular, $B[z] = \sum_{j = 1}^J s_j$.
\end{corollary}

\medskip

Our second local stability result states that dislocations are stable
provided they are sufficiently distant from any domain boundary. To
state this result, let $\Omega := \{ \xi \in \L \sep \xi_2 \leq 0\}$
be a discrete half space and let $\Bonds^\Omega := \{ b =
(\xi,\zeta) \in \Bonds \sep \xi,\zeta \in \Omega\}$ be the
corresponding set of bonds.

\begin{corollary}[Domain with boundary]
  \label{th:boundaries}
  Suppose that {\bf (STAB)} holds.

  Let $C \in \Cells$ such that $L := - (x^C)_2 > 0$. If $L$ is
  sufficiently large, then there exists a locally stable half-space
  configuration $z : \Omega \to \R$ containing a dislocation. That is,
  for any $\alpha \in [Dz]$,
  \begin{enumerate}
  \item $\Cores[\alpha] \subset x^C + A$,
  \item $\int_{\partial (x^C+A)} \alpha = 1$, and
  \item there exists $\epsilon > 0$ such that
    \begin{displaymath}
      \sum_{b \in
        \Bonds^\Omega} \B( \psi(Dz_b+Dv_b) - \psi(Dz_b) \B) > 0 \qquad
      \forall v \in \Usz, \|Dv\|_2 < \epsilon.
    \end{displaymath}
  \end{enumerate}
\end{corollary}

\subsection{Regularity}
\label{sec:regularity}
The globally stable equilibrium configuration $y$, whose existence we
established in Theorem \ref{th:ex_nec_disl} is of the form $y = \yh +
u$, where $\yh(\xi) = \frac{1}{2\pi} \arctan(\frac{\xi_2}{\xi_1})$ is
the continuum linearised elasticity solution for a screw
dislocation. We refer to \S\ref{sec:Ediff} and in particular to
Theorem \ref{th:minim_E} for further details.

This fact implies that only a finite amount of energy is stored in the
dislocation core, and that, up to some fixed prescribed error
tolerance, the linearised elasticity displacement field is accurate
outside of some fixed radius. These observations give rise to new
points of view on the concepts of dislocation {\em core energy} and
{\em core radius}, which we explore in \cite{core}.  In particular
the core radius is an interesting concept related to
the decay of the ``corrector'' $u$ to the configuration $\yh$
predicted by linear elasticity. Here, we state a {\em regularity
  result} proven in more general form in \cite{defects}, which
precisely quantifies the rate of decay of $Du$. In effect, the results
states that the decay of $Du$ is the same as predicted by linearised
elasticity.

\begin{proposition}
  Let $y = \yh + u, u \in \Hsi$ be a locally stable equilibrium, then
  there exists $C$ such that
  \begin{equation}
    \label{eq:decay_rate}
    |Du_b| \leq C\d{b}^{-2} \qquad \forall b \in \Bonds.
  \end{equation}
\end{proposition}

\begin{remark}
  One may expect, and numerical simulations confirm this, that the
  corrector $u$ satisfies the three-fold symmetry of the lattice $\L$
  with respect to its origin (recall that the origin lies in the
  barycentre of a cell). Exploiting this symmetry, one can observe
  that the decay rate is in fact $|Du_b| \leq C \d{b}^{-4}$. However,
  as soon as the symmetry is broken, for example by applying a small
  shear displacement at infinity, or by moving the core off the centre
  of the cell, the generic rate \eqref{eq:decay_rate} is observed
  also numerically.
\end{remark}

\section{Analysis of the energy difference functional}
\label{sec:Ediff}

\subsection{Extension of the energy difference functional}
\label{sec:ext_Ediff}
We fix a displacement $\yh$ and define the functional $\E(u) :=
E(\yh+u; \yh)$. For $u \in \Usz$ this is always well-defined.  If
$D\yh_b \in \R \setminus (\Z + B(\epsilon))$ for all $b \in \Bonds$,
where $\epsilon > 0$, then the first and second variations (in the
sense of directional derivatives) are also well-defined, and given by
\begin{align}
  \label{eq:defn_delE}
  \< \del\E(0), v \> &= \sum_{b \in \Bonds} \psi'(D\yh_b)
  \cdot Dv_b, \quad \text{ for } v \in \Usz, \quad \text{and} \\
  \< \ddel\E(0) v, w \> &= \sum_{b \in \Bonds} \psi''(D\yh_b)
  \cdot Dv_b Dw_b \quad \text{for } v,w \in \Usz.
\end{align}
$\ddel\E(0)$ can clearly be extended by continuity to $v, w \in \Hsi$,
but this is less obvious for $\del\E(0)$ or for $\E$ itself. We first
state a general result.

\begin{lemma}
  \label{th:abstract_E_extension}
  Let $\yh \in \Us$ satisfy $D\yh_b \in \R \setminus
    (\Z+B(\epsilon))$ for some $\epsilon > 0$ and suppose that
  $\del\E(0)$ is a bounded linear functional ($\< \del\E(0), v \> \leq
  C \| D v \|_2$ for all $v \in \Usz$).  Then, $\E : \Usz \to \R$ is
  continuous with respect to the norm $\|D \cdot \|_2$; hence, there
  exists a unique continuous extension of $\E$ to $\Hsi$.
\end{lemma}
\begin{proof}
  The proof of this result is analogous to the proof of Theorem 2.8
  (ii) in \cite{OrtnerTheil:CauchyBorn}. For convenience we give a
  brief outline.

  For $u \in \Usz$ it is easy to see that
  \begin{align*}
    \E(u) &= \sum_{b \in \Bonds} \b[ \psi(D\yh_b + Du_b) -
    \psi(D\yh_b) - \psi'(D\yh_b) Du_b \b] + \sum_{b \in \Bonds}
    \psi'(D\yh_b) Du_b.
    % \\
    % %
    % &= \< \del\E(0), u \> + \int_0^1 (1-t) \b\< \ddel\E(t u) u, u
    % \> \dt.
  \end{align*}
  Since we assume that $\del\E(0)$ is a bounded functional, the
  second term on the right-hand side is
  continuous. Using the fact that $\| Dw\|_\infty \leq \| D w
    \|_2$, the smoothness of $\psi$, and the fact that each summand in
    the first group is effectively quadratic in $Du_b$, it is easy to
  show that the second term on the right-hand side is continous as
  well.
\end{proof}

For future reference, we now derive a simple condition on $\yh$ under
which $\del\E(0)$ is a bounded functional.  Applying summation by
parts to \eqref{eq:defn_delE} we obtain
\begin{equation}
  \label{eq:delE_frc}
  \< \del\E(0), v \> = \sum_{\xi \in \L} f(\xi) \cdot v(\xi), \quad
  \text{where} \quad f(\xi) := \sum_{b\in\Rg_\xi}\psi'(D\yh_b(\xi))
\end{equation}
is the force acting on atom $\xi$ under the displacement $\yh$. The
following result states that, if $\yh$ is sufficiently close to
equilibrium in the far-field, then $\del\E(0)$ is a bounded linear
functional.

\begin{lemma}
  \label{th:delE_farfield_equ}
  Suppose that a displacement $\yh$ has associated forces $f(\xi)$
  satisfying the bound $|f(\xi)| \leq C_1 (1+|\xi|)^{-t}$ for some $t
  > 2$, then $\< \del\E(0), v \> \leq C_2 \| D v \|_2$ for all
  $v \in \Usz$.
\end{lemma}
\begin{proof}
  Proposition 12 in \cite{OrtSha:interp:2012} immediately implies that
  \begin{equation}
     \b\| \smfrac{v}{\log(|\xi|+2)} \b\|_\infty
    \leq C  \| D v \|_2,\label{eq:BMO.est}
  \end{equation}
  for some constant $C > 0$. (This inequality is essentially a
  consequence of the embedding $\| v \|_{\rm BMO} \leq C \| \nabla v
  \|_{L^2}$ for $v \in C^1(\R^2)$.)
  
  We can therefore estimate
  \begin{align*}
    \b|\< \del\E(0), v\>\b| \leq \sum_{\xi \in \L} |f(\xi)|\, |v(\xi)|
    \leq  \b\| \log(|\xi|+2) f \b\|_1 \, \b\|
    \smfrac{v}{\log(|\xi|+2)} \b\|_\infty.
  \end{align*}
  The assumption $|f(\xi)| \leq C |\xi|^{-t}$ with $t > 2$ implies
  that $\| \log(|\xi|+2) f \|_1$ is finite. 
\end{proof}

% It is now straightforward to extend the energy functional $\E$ to
% $\Hsi$ as well.

% \begin{lemma}
%   \label{th:E_farfield_equ}
%   Under the conditions of Lemma \ref{th:delE_farfield_equ}, $\E : \Usz
%   \to \R$ is continuous with respect to the norm $\| D \cdot
%   \|_{\ell^2}$. In particular, there exists a unique continuous
%   extension of $\E$ to $\Hsi$.
% \end{lemma}
% \begin{proof}
%   With the key technical result given in Lemma
%   \ref{th:delE_farfield_equ}, 
% \end{proof}

\subsection{The reference displacement}
\label{sec:ref_def}
We now specify the reference displacement $\yh$ used in the definition
of $\E$ in \S\ref{sec:ext_Ediff}. It is best to think of $\yh$ as
prescribing a far-field boundary condition $y(\xi) \sim \yh(\xi)$ as
$|\xi| \to \infty$. We wish to choose $\yh$ in such a way that it
enforces a geometrically necessary dislocation, and at the same time
satisfies the condition of Lemma \ref{th:delE_farfield_equ}. 

A natural choice is the dislocation displacement field from linear
elasticity theory.
Since it is instructive (though not essential to our proofs) we give a brief
motivation of this construction. In the far-field, we expect that continuum
linearized elasticity theory is a good approximation to the atomistic
equilibrium condition $\del\E(0) =0$. This can be formalized by first
deriving the Cauchy--Born approximation and then linearising it. Due to
the hexagonal symmetry of $\L$ one finds that the linearised continuum
approximation is simply Laplace's equation, $\Delta \yh(x) = 0$.

\begin{figure}
  \includegraphics[height=5cm]{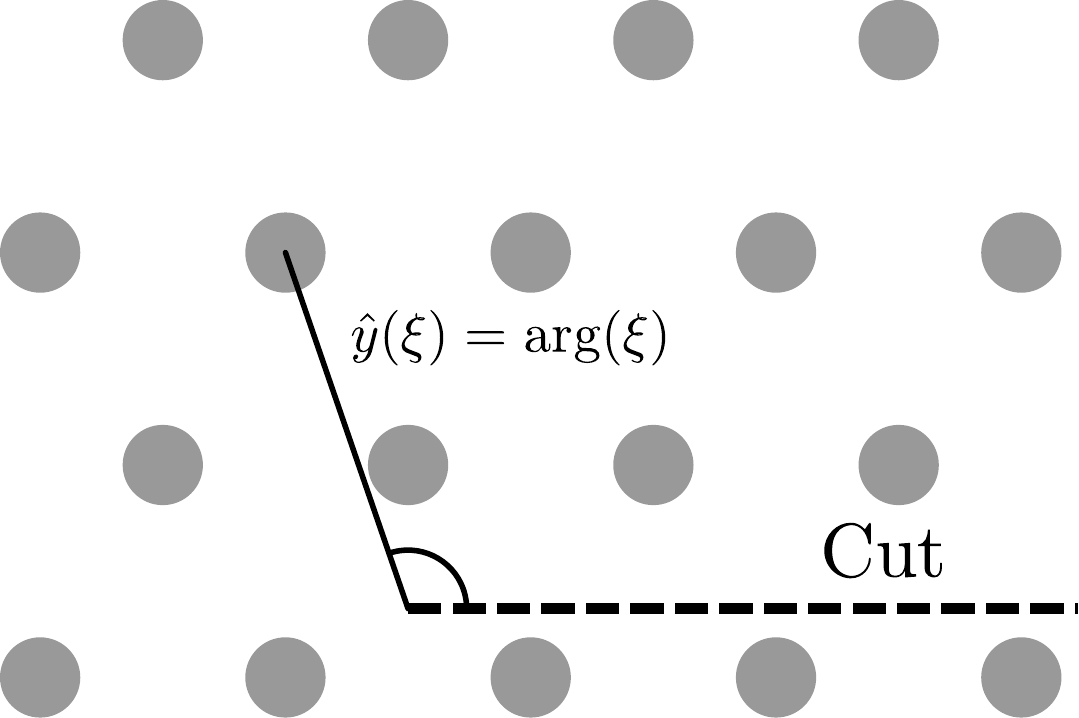}
  \caption{An illustration of the coordinate system and reference configuration
  chosen.}
  \label{fig:hatydefn}
\end{figure}

Hence, following Section 3-2 in Hirth \& Lothe \cite{HirthLothe}, we define
$\yh$ as follows:
\begin{equation}
  \label{eq:defn_yh}
  \yh(x):=\smfrac{1}{2\pi}\arg(x)
  =\smfrac1{2\pi}\arctan\b(\smfrac{x_2}{x_1}\b),
\end{equation}
where we identify $x \in \R^2$ with the point $x_1+i x_2 \in
\C$, and the branch cut is taken along the positive $\xi_1$-axis, as shown in
Figure \ref{fig:hatydefn}.

The gradient (away from the branch cut) is given by
\begin{equation}
  \label{eq:defn_Dyh}
  \D\yh(x) = \Big(\frac{-x_2}{2\pi r^2},
  \frac{x_1}{2\pi r^2}\Big)^T,
\end{equation}
where $r := |x|$. This function can be extended to a function in
$C^\infty(\R^2 \setminus \{0\})$, which we take as the {\em
  definition} of $\D\yh$ from now on. Moreover, we can check that
indeed $\Delta \yh(x) = {\rm div} (\nabla \yh(x)) = 0$, in the pointwise
sense, for $x \neq 0$.

Let $\alh = (\alh_b)_{b \in \Bonds}$ be a bond-length 1-form
associated with $\yh$; we claim this is unique, and the following lemma provides
a convenient formula for $\alh_b$ in terms of $\D \yh$.

\begin{lemma}
  \label{th:alhb_intDyh}
  Let $\alh \in [D\yh]$ then for any bond $b=(\xi,\xi+a_i) \in
  \Bonds$, we have
  \begin{equation}
    \label{eq:alhb_intDyh}
    \alh_b = \int_0^1 \D\yh\big(\xi+ta_i\big)\cdot a_i \dt.
  \end{equation}
\end{lemma}
\begin{proof}
  By definition, $\D\yh$ is independent of the choice of branch
  cut. Moreover, if the branch cut is chosen differently, then the
  displacement at each site is only changed by an integer, which means
  $\alh$ does not change; hence, $\alh$ is also independent of the
  branch cut.

  Now fix $b = (\xi,\xi+a_i) \in \Bonds$. Since the origin lies at the
  centre of a cell we can redefine $\yh$ with a branch cut that does
  not intersect $b$. The Fundamental Theorem of Calculus gives
  \begin{equation}
    \int_0^1 \D\yh\big(\xi+ta_i\big)\cdot a_i \dt = \yh(\xi+a_i)-\yh(\xi)
    =\smfrac1{2\pi}\b(\arg(\xi+a_i)-\arg(\xi)\b),
  \end{equation}
  and since we have assumed that $b$ is a nearest neighbour bond, it
  has length 1. The term on the right hand side is $1/2\pi$ times the
  angle formed by the points $\xi$, $0$ and $\xi+a_i$, which is
  maximised by making $\xi$ and $\xi+a_i$ as close to the origin as
  possible --- that is, when $\xi$ and $\xi+a_i$ are on the boundary of $C_0$.
  It follows that the angle can be no larger than $\smfrac{2\pi}3$, and hence
  $\alh_b=\yh(\xi+a_i)-\yh(\xi)\in[-\smfrac13,\smfrac13]$. This
  implies that $\alh$ is unique, since $D\yh_b\neq\pm\smfrac12$ for
  all $b\in\Bonds$.
\end{proof}

As an immediate corollary of Lemma \ref{th:alhb_intDyh} we obtain
the following bound on $\alh_b$:
\begin{align}
  \label{eq:alpha.hat.bnd}
  |\alh_b| \leq \frac1{2\pi\d{b}} \qquad \forall b \in \Bonds.
\end{align}
We conclude the analysis of $\yh$ by showing that it satisfies the
conditions of Lemma \ref{th:delE_farfield_equ}.

\begin{lemma}
  \label{th:decay_frc_yh}
  Let $\yh$ be defined by \eqref{eq:defn_yh}, and let $f(\xi), \xi \in
  \L$, be the associated forces (see \eqref{eq:delE_frc}), then
  $\alh\in[D\yh]$ satisfies $\alh_b \in [-1/3, 1/3]$ and 
  \begin{displaymath}
    |f(\xi)| \lesssim |\xi|^{-3} \qquad \forall \xi \in \L.
  \end{displaymath}
  In particular $\yh$ satisfies all conditions of Lemma
  \ref{th:delE_farfield_equ}.
\end{lemma}
\begin{proof}
  Recall from \eqref{eq:delE_frc} that
  \begin{displaymath}
    f(\xi) = \sum_{b \in \Rg(\xi)}
    \psi'(D_{b} \yh) = \sum_{b \in \Rg(\xi)}
    \psi'(\alh_b).
  \end{displaymath}
  Taylor expanding $\psi'_{b}$ to third order, using the fact that
  $\psi'(0) = \psi'''(0) = 0$ (since $\psi$ is even about $0$),
  gives
  \begin{align*}
    f(\xi) = \sum_{b\in\Rg(\xi)} \B[ \psi''(0) \alh_b +
    \smfrac16 \psi^{(4)}(s_b)(\alh_b)^3 \B],
  \end{align*}
  for some $s_b \in {\rm conv}\{0, \alh_b\}$. Applying
  \eqref{eq:alpha.hat.bnd} we obtain
  \begin{equation}
    \label{eq:decay_frc_yh:10}
    f(\xi) = \sum_{b\in\Rg(\xi)}\psi''(0) \alh_b + O\b(\d{b}^{-3}\b),
  \end{equation}
  
  We now inspect the sum on the right-hand side of
  \eqref{eq:decay_frc_yh:10} in more detail. Applying
  \eqref{eq:alhb_intDyh} we rewrite this sum as
  \begin{displaymath}
    \sum_{b\in\Rg(\xi)}\psi''(0)\alh_b = \sum_{i =
      1}^6 \psi''(0) 
    \int_0^1 \D\yh\big(\xi+ta_i\big) \cdot a_i \dt.
  \end{displaymath}
  Taylor expanding $\D\yh\big(\xi+ta_i\big)$ and using the fact that
  $|\D^4\yh(x)| \lesssim |x|^{-4}$, we obtain
  \begin{displaymath}
    \int_0^1 \D\yh\big(\xi+ta_i\big) \cdot a_i \dt = \D\yh(\xi) \cdot
    a_i + \smfrac12 \D^2 \yh(\xi) [a_i,a_i] + \smfrac16 \D^3 \yh(\xi)
    [a_i,a_i,a_i] + O\b(|\xi|^{-4} \b).
  \end{displaymath}
  Summing over $i = 1, \dots, 6$ the first and third terms cancel
  since $a_{i+3} = -a_i$, hence we obtain
  \begin{equation}
    \label{eq:decay_frc_yh:20}    
    \sum_{b\in\Rg(\xi)} \psi''(0) \alh_b = \frac12 \sum_{i = 1}^6 a_i^T \D^2 \yh(\xi) a_i + O\b(|\xi|^{-4}\b).
  \end{equation}
  We now observe that
  % (cf. Lemma ***) \cco{though
  %   actually this is a well-known result; we should probably cite
  %   something}
  \begin{displaymath}
    \frac12 \sum_{i = 1}^6 a_i^T \D^2 \yh(\xi) a_i = -\smfrac32 \Delta
    \yh(\xi) = 0.
  \end{displaymath}
  Inserting the last identity into \eqref{eq:decay_frc_yh:20} and
  combining the resulting estimate with \eqref{eq:decay_frc_yh:10} we
  obtain the stated estimate on $|f(\xi)|$.
\end{proof}

\subsection{The variational problem in \texorpdfstring{$\Hsi$}{W12}}
\label{sec:var_problem}
Combining Lemma \ref{th:decay_frc_yh} with Lemma
\ref{th:abstract_E_extension} and Lemma \ref{th:delE_farfield_equ}, we
deduce that $\E(u) := E(\yh + u; \yh)$ is a well-defined and
continuous functional on $\Hsi$, where $\yh$ is the reference
configuration defined in \eqref{eq:defn_yh}. It will later be
convenient to recall from the proof of Lemma
\ref{th:abstract_E_extension} that the explicit definition of the
extension is
\begin{equation}
  \label{eq:explicit_E_extension}
  \E(u) = \sum_{b \in \Bonds} \b[ \psi(\alh_b + Du_b) - \psi(\alh_b)
  - \psi'(\alh_b) Du_b \b] + \< \del\E(0), u \>,
\end{equation}

In the next section, \S~\ref{sec:mainproof}, we will prove the following
result:

\begin{theorem}
  \label{th:minim_E}
  There exists $u \in \Hsi$ such that $\E(u) \leq \E(v)$ for all $v
  \in \Hsi$.
\end{theorem}

\medskip As an immediate corollary we can now prove Theorem
\ref{th:ex_nec_disl}. 
% The intuition is that $u$ has finite energy and
% hence cannot have a non-zero net Burgers vector. Thus, $y := \yh + u$
% must have the same net Burgers vector as $y$.

\begin{proof}[Proof of Theorem \ref{th:ex_nec_disl}]
  Let $y := \yh + u$, where $u$ is a minimizer of $\E$ in
  $\Hsi$. Since $Du \in \ell^2(\Bonds)$ it follows that $|Du_b| \to 0$
  uniformly as $\d{b} \to \infty$. Using also the fact that $|\alh_b|
  \to 0$ uniformly (cf. \eqref{eq:alpha.hat.bnd}), we conclude that
  $\alpha_b = \alh_b + Du_b+z_b$, where $z_b$ is a compactly
  supported, integer-valued 1-form. From the definition of the net
  Burgers vector and from \eqref{eq:B[y]_ind_of_alpha}, it now follows
  immediately that $B[y] = B[\yh] = 1$. Moreover, minimality of $u$
  implies that $y$ is a globally stable equilibrium in the sense of
  Definition~\ref{def:defn_stable_equilib}. 
\end{proof}

Theorem \ref{th:minim_E} is interesting in its own right: it shows
that atomistic configurations containing dislocations can be obtained
as global minimizers of a variational problem formulated over
$\L$. This is particularly useful for further study (e.g., of
regularity; cf. \S~\ref{sec:regularity}) of dislocations in this
model.

We also remark that any local minimizer $u$ of $\E$ in $\Hsi$ would
give rise to a locally stable equilibrium with net Burgers vector
$B(\yh + u) = 1$. The advantage of local minimisers is that they can
be computed numerically.

\section{Proof of Theorem \ref{th:minim_E}}
\label{sec:mainproof}
As currently formulated, it is not obvious that the energy $\E$ is bounded
below, and it is even less clear whether $\E$ is coercive in a sense which would
allow us to invoke the Direct Method. This is due in large part to the fact that
the reference configuration is nonlinear and $\psi$ is periodic, so the
integrand has infinitely many energy wells.

The periodicity of $\psi$ allows the creation of dislocation dipoles
`cheaply'. If
dipoles are well-separated, then each dipole gives a positive
contribution to the energy which is proportional to the logarithm of
the dipole length (the separation distance between the two cores of
the dipole). However, for generic configurations of dipoles the sign
of the energy contribution is difficult to determine, since it depends
strongly upon the relative orientations of the dipoles. In essence,
this is a geometric nonlinearity of the system, and most of the effort
expended in what follows will be to control the number of dipoles that
can form.

From a technical point of view the issue arises as follows: in
\S\ref{sec:bonds} we decomposed $Dy = \alpha+w$, $\alpha \in [Dy]$,
since the energy of the displacement $y$ only depends
on $\alpha$ due to the periodicity of the potential
$\psi$. Consequently, if we have a sequence $u^n$ with $\E(u^n)$
uniformly bounded, then this will bound only $\| \beta^n \|_2$
for $\beta^n \in [Du^n]$, and not $\|Du^n\|_2$. In particular,
generic minimising sequences cannot be weakly compact.

By exploiting the vertical shift invariance (\S\ref{sec:dips_branch})
and the horizontal translation invariance (\S\ref{sec:orig.shift}) of
the energy $\E$, we will {\em construct} a weakly compact minimising
sequence. Having made this special choice of minimising sequence, we
use a {\em profile decomposition} in \S\ref{sec:existence}. We show
that each profile obtained in this way has net Burgers vector zero, leading
to the conclusion that the net Burgers vector of the limit remains
$1$, and proving existence of a minimiser with the properties
required.

\subsection{An elementary lower bound}
\label{sec:lower_bound}
Our eventual goal is to establish a coercivity result for $\E$. We
begin with an elementary lower bound that will motivate subsequent
constructions.

Let $y = \yh + u, u \in \Hsi$, be a trial displacement, $\alpha \in
[Dy]$, and recall that $\alh = [D\yh]$ is unique. Since $u\in\Hsi$,
and hence $Du_b \to 0$ as $\d{b} \to \infty$, it follows that $y$ has
a well-defined net Burgers vector in the sense of Definition
\ref{def:netB}, and $B[y]=B[\yh]=1$.

Let
\begin{equation}
  \label{eq:defn_beta}
  \beta := \alpha - \alh;
\end{equation}
this 1-form satisfies the property that
\begin{equation*}
  \int_{\partial C} \beta \neq 0\qquad\text{if and only if}\qquad
    \int_{\partial C}\alpha \neq \int_{\partial C} \alh,
\end{equation*}
that is, dislocation cores present in $\beta$ are those that are
introduced by the addition of $u$ to $\yh$.

\begin{remark}
  We note that $\beta_b$ does not necessarily belong to $[-1/2, 1/2]$,
  and hence is not a bond length 1-form, so the definitions of
  \S\ref{sec:burgers_vec} do not strictly apply; however, it remains
  a 1-form in the sense of \cite[\S3.1]{ArizaOrtiz05}. As
  $\int_{\partial C} \beta \in \{0, 1, -1\}$ for all $C \in \Cells$,
  we shall therefore slightly abuse our notation and refer to
  dislocation cores in $\beta$ as the cells $C\in\Cells$ for which
  \begin{equation*}
    \int_{\partial C}\beta = \pm 1.
  \end{equation*}
  We also define $\Coresp[\beta]$, $\Coresm[\beta]$ and
  $\Cores[\beta]$ in the obvious way.
\end{remark}

\medskip

Next, we define $z : \Bonds \to \Z$ via
\begin{equation}
  \label{eq:defn_Du_beta_z}
  Du=\beta+z,
\end{equation}
which is compactly supported since $\beta, Du\in\ell^2(\Bonds)$.  We
shall see in \S\ref{sec:dips_branch} that the support of $z$ can be
thought of as a union of branch cuts connecting dislocation dipoles.

With this notation, we obtain the following result.

\begin{lemma}
  For any $u\in\Hsi$ with $Du=\beta+z$ as in \eqref{eq:defn_Du_beta_z} and for
  any $\eps>0$, we have 
  \begin{equation}
    \label{eq:lb:0}
    \E(u) \geq \b(\smfrac12\psi^{\prime\prime}(0)-\eps\b) \|\beta\|_{\ell^2}^2
      - \sum_{b \in \Bonds} \psi'(\alh_b) z_b + \< \del\E(0), u \> -C_\eps,
  \end{equation}
  where $C_\eps>0$ is a constant that is independent of $u$.
\end{lemma}

\begin{proof}
This estimate arises from the expression \eqref{eq:explicit_E_extension}; using
the periodicity of the potential $\psi$, we can write
\begin{displaymath}
  \E(u) = \sum_{b\in\Bonds} \b(\psi(\alh_b+\beta_b)-\psi(\alh_b)
    -\psi^\prime(\alh_b)\beta_b\b) -\sum_{b\in\Bonds}\psi'(\alh_b )z_b +\< \del\E(0), u \>.
\end{displaymath}
Define the function
 \begin{align*}
    g(s,t):= \cases{\displaystyle
      \frac{\psi(t+s)-\psi(t)-\psi^\prime(t)s}{s^2} &s\neq0,\\
      \smfrac12 \psi''(t) &s=0.
    }
  \end{align*}
  By assumption $(\psi5)$ in \S\ref{sec:Ediff_defn},
  $g(s,0)\geq\smfrac12\psi^{\prime\prime}(0)$ for any $|s|\leq1/2$.

  Since $g$ is uniformly continuous on $[-1/2, 1/2] \times [-\tau, \tau]$
  for some $\tau > 0$, it follows that for each $\eps > 0$ there exists
  $\delta(\eps) > 0$ such that
  \begin{align*}
    g(s,t) \geq \smfrac12\psi^{\prime\prime}(0)-\eps\quad\text{for} \quad|s|
    \leq\smfrac12+\del(\eps) \text{ and } |t|\leq \delta(\eps).
  \end{align*}

% Since $\psi$
% is continuous, we have further that for any $\eps>0$, there exists $\delta(\eps)>0$
% such that
% \begin{equation*}
%   g(s,0)\geq \smfrac12\psi^{\prime\prime}(0)-\smfrac12\eps\qquad\text{when }|s|
%     \leq\smfrac12+\del(\eps).
% \end{equation*}
% Noting that $g$ is continuous, and hence uniformly continuous on compact sets, it
% follows that for any $\eps>0$, there exists $\del'(\eps)>0$ such that
% \begin{align*}
%   g(s,t)\geq\smfrac12\psi^{\prime\prime}(0)-\eps\quad\text{when}\quad|s|
%     \leq\smfrac12+\del(\eps),\;|t|\leq\del'(\eps).
% \end{align*}
  Next, we note that \eqref{eq:alpha.hat.bnd} implies
\begin{equation*}
  |\alh_b| \leq \frac{1}{2 \pi \d{b}} \quad \text{and} \quad
  |\beta_b| = |\alpha_b-\alh_b| \leq \frac12+\frac{1}{2\pi\d{b}}.
\end{equation*}
Hence there exists $R_0 > 0$ such that, for $\d{b} \geq R_0$, 
\begin{displaymath}
  g(\beta_b, \alh_b) \geq \smfrac12\psi^{\prime\prime}(0)-\eps,
\end{displaymath}
which can equivalently be stated as
\begin{equation}
  \psi(\alh_b+\beta_b)-\psi(\alh_b)-\psi^\prime(\alh_b)\beta_b\geq
  \b(\smfrac12 \psi^{\prime\prime}(0)-\eps\b)|\beta_b|^2 \qquad
  \text{for } \d{b} \geq R_0.
  \label{eq:psi_lb}
\end{equation}

It may be checked that
\begin{displaymath}
  \#\b\{b\bsep \d{b} < R_0\b\}\lesssim R_0^2,
\end{displaymath}
and since $\psi$, $\psi'$ and $\beta$ are uniformly bounded, it
therefore follows that
\begin{displaymath}
  \sum_{b\in\Bonds} \B(\psi(\alh_b+\beta_b)-\psi(\alh_b)
    -\psi^\prime(\alh_b)\beta_b\B) \geq \b(\smfrac12\psi^{\prime\prime}(0)-\eps\b)
    \sum_{b\in\Bonds}|\beta_b|^2- C R_0^2. \qedhere
\end{displaymath}
\end{proof}
\medskip

We can think of $\|\beta\|_{2}^2$ as estimating elastic stored
energy. In the following sections we will establish several results on
$z = Du - \beta$, which will eventually allow us to bound the
remaining terms $\<\del\E(0),u\>$ and $\sum_b\psi'(\alh_b)z_b$ in
\eqref{eq:lb:0}.

\subsection{Dipoles \& Branchcuts}
\label{sec:dips_branch}
Let $y = \yh + u$, $u \in \Hsi$, be a trial displacement, $\alpha \in
[Dy]$, and let $\beta, z$ be defined by
\eqref{eq:defn_Du_beta_z}. While $\alpha$ and hence $\beta$ are
uniquely defined (except in borderline cases when $\alpha_b \in \{\pm
1/2\}$), one can exploit the vertical shift invariance of the lattice
(encoded in assumption ($\psi$1), periodicity of $\psi$) to construct
equivalent displacements $\tilde{u}\in\Hsi$,
\begin{equation}
  \tilde{u}:=u+U \label{eq:vert.shifts}
\end{equation}
where $U : \L \to \Z$ and $U\in\Usz$, and hence modify the $z$
component.

If we let $\tilde{y} := \yh + \tilde{u}$, then clearly, $\alpha \in
[D\tilde{y}]$ and this leads to the same definition of
$\beta$. Crucially, though, $D\tilde{u} - \beta \neq D u - \beta$. We
can therefore ask how to choose $U$ in an ``optimal'' way. It turns
out that minimizing the total length of the branch cuts is a useful
choice, which amounts to minimizing
$\|Du+DU-\beta\|_1=\|z+DU\|_1$. Since $z$ has compact support, a
minimizer clearly exists, but it need not be unique; see Figure
\ref{fig:typ_conf}. 
We may therefore assume, without loss of generality, that $u$
satisfies the {\em discrete minimal connection property} ({\DMCP})
\begin{equation}
  \label{eq:DMCP}
  \|Du-\beta\|_1=\|z\|_1=\min_{Z : \L \to \Z}\,\|Du+DZ-\beta\|_1.
\end{equation}
% We will say $u$ satisfies the {\em weak geometric constraint} if it
% satisfies \eqref{eq:DMCP}. 
This minimality condition is similar to the idea of {\em minimal
  connections}, introduced in \cite{BCL86}.

\medskip

We will now establish various properties of the structure of $z$
defined in~\eqref{eq:defn_Du_beta_z}. In particular, we will show that
$z$ can be decomposed into a sum $\sum z^m$ and that the support of
each $z^m$ is analogous to a branch cut for a dipole. 

% Furthermore, we show that by choosing $u$ to satisfy the minimality
% property \eqref{eq:DMCP}, it will consequently have (or can be
% chosen to have) a series of other properties which allow us to
% construct good estimates on $\E(u)$.

% The first fact we will use is that we can decompose $z$ as a sum of functions
% which have support only on bonds connecting pairs of cores of oppositely
% signed Burgers vector.

\begin{lemma}
  \label{lem:z_decomp}
  Let $u \in \Hsi$ satisfy the \DMCP~\eqref{eq:DMCP}
  and suppose $Du=\beta+z$ as in \eqref{eq:defn_Du_beta_z}. Then we can write
  \begin{equation}
    z=\sum_{m=1}^{M}z^m,\label{eq:z_decomp}
  \end{equation}
  where $M = \#\Coresp[\beta]$ is the number of dipoles contained in
  $\beta$ and $z^m : \Bonds \to \{-1,0,+1\}$, $m = 1, \dots, M$,
  satisfy the following properties:
  \begin{enumerate}
    \item $z_{b_i}^m=1$ on a sequence of bonds $(b_i)_{i=0}^n$ such that
      \begin{enumerate}
        \item $\partial b_i$ and $\partial b_{i+1}$ share a common 0-cell for 
          each $i=0,\ldots,n-1$,
        \item $b_0\in\partial C_m^-$ and $-b_n\in\partial C_m^+$ where
          $C_m^+\in\Coresp[\beta]$ and $C_m^-\in\Coresm[\beta]$.
      \end{enumerate}
    \item $z_b^m=0$ for bonds outside the set $\{\pm b_i\sep i=0,\ldots n\}$.
  \end{enumerate}
\end{lemma}
\begin{proof}
  The result is geometrically intuitive; see Figure
  \ref{fig:typ_conf}. We therefore postpone a complete proof to
  Appendix \ref{app:z_decomp}.
\end{proof}

\begin{figure}
  \includegraphics[height=8cm]{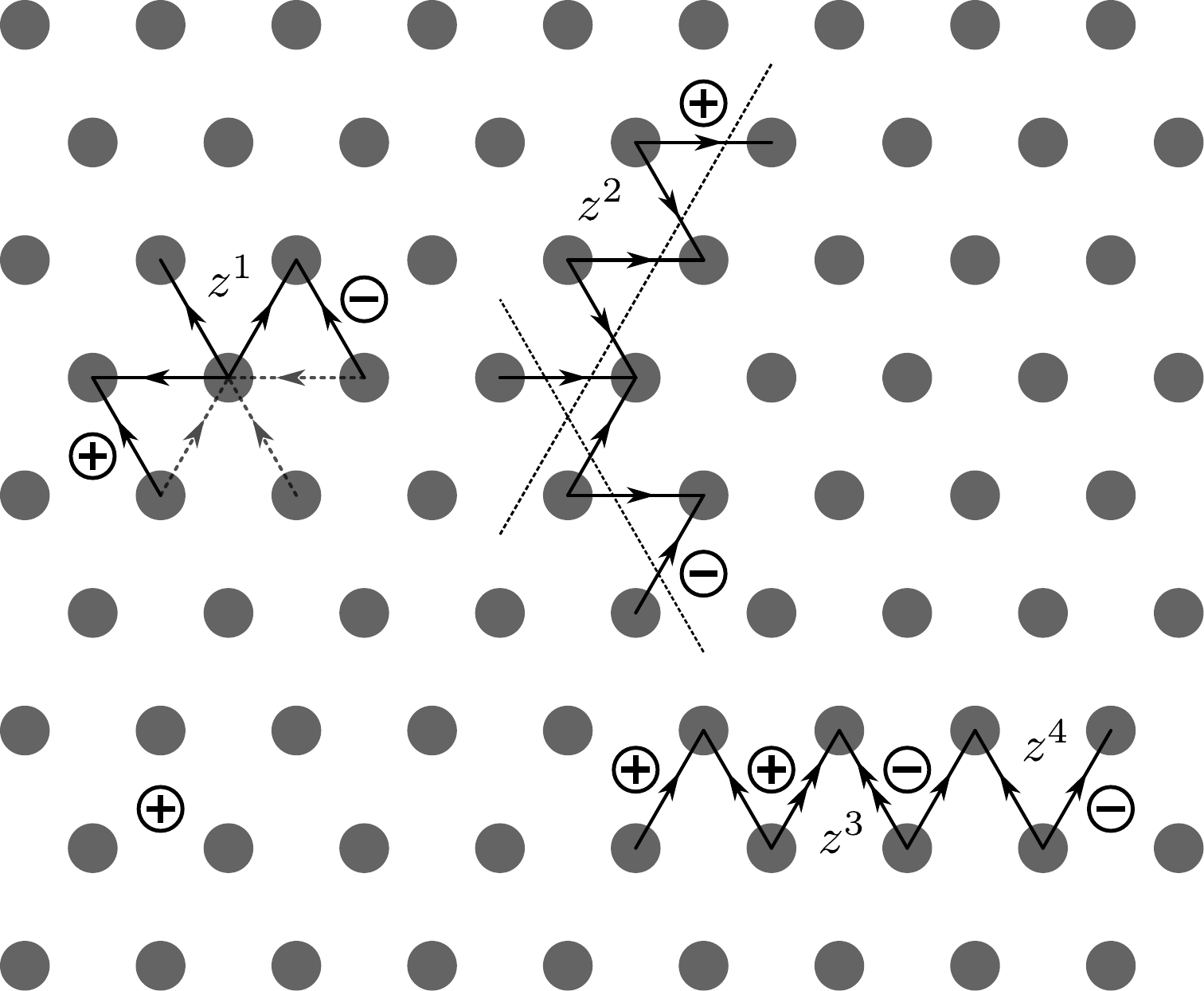}
  \caption{A typical example of the support of a minimal $z$ and its
    decomposition into $z^m$ for a given distribution of dipoles. An
    arrow means that $z_b=1$ on the bond pointing in the direction of
    the arrow, and a double arrow means $z_b=2$. The grey dashed bonds
    on the left of the diagram show an alternative definition of $z$
    with the same minimal norm, and the black dashed lines show two
    lines which show that $z^2$ can be decomposed into 2 straight
    cuts.}
  \label{fig:typ_conf}
\end{figure}

% \medskip
% The proof of this lemma may be found in Appendix \ref{app:cut_proofs}. In Figure
% \ref{fig:typ_conf}, we show a typical minimal $z$, noting that in general the
% minimal $z$ is non-unique (see the example of the left-hand side).

We will say that $z^m$ {\em connects} the dislocation cores $C^+_m$
and $C^-_m$.  Moreover, we obtain the following corollary.

\begin{corollary}
  \label{th:short_dual_latt}
  Each 1-form $z^m$ in the decomposition \eqref{eq:z_decomp} can be
  identified with a shortest path in the dual lattice between
  $C^+_m\in\Coresp[\alpha]$ and $C^-_m\in\Coresm[\alpha]$, and further
  \begin{equation*}
    \|z^m\|_1 = \hop_2(C^+_m,C^-_m).
  \end{equation*}
\end{corollary}
\begin{proof}
  First, we observe that, due to the \DMCP~\eqref{eq:DMCP}
  and the decomposition proven in Lemma~\ref{lem:z_decomp},
  \begin{equation}
    \| z \|_1 = \sum_{m = 1}^M \| z^m \|_1. \label{eq:length_decomp}
  \end{equation}
  That is, if $z^m_b, z^{m'}_b \neq 0$, then $z^m_b, z^{m'}_b$ have
  the same sign.

  The construction employed in the proof of Lemma \ref{lem:z_decomp}
  identifies a sequence of bonds $b_i$ and cells $C_{m,i}$ such that
  $b_i,-b_{i+1}\in\partial C_{m,i}$, $z^m_{b_i}=1$. Using the natural
  identification of cells with points in the dual lattice, this
  implies that $C_{m,i}$ are adjacent in the dual lattice, and
  furthermore that $b_i$ can be identified with edges connecting these
  cells; this leads to the fact that
  \begin{align*}
    \|z^m\|_1 &\geq \hop_2(C^+_m,C^-_m).
  \end{align*}

   To prove the converse, we take the path in the dual lattice
    corresponding to $z^m$, adjoin a shortest path between $C_m^+$ and
    $C_m^-$ in the dual lattice, and thus obtain a closed dual
  lattice path. We construct a polygonal closed path in $\R^2$ by
  connecting the barycentres of the cells along the path and define
  $U$ to be the characteristic function of the bounded interior of
  this loop in $\R^2$.
  
  It is now straightforward to check that by defining $\tilde{z}^m:=z^m+DU$
  and $\tilde{z}:=z-z^m+\tilde{z}^m$
  \begin{displaymath}
    \| \tilde{z} \|_1 \leq \sum_{m' = 1}^M \| z^{m'} \|_1 - \| z^m\|_1
    + \| \tilde{z}^m \|_1.
  \end{displaymath}
  Since $U$ is a compactly supported integer shift as in \eqref{eq:vert.shifts},
  the discrete minimal connection property \eqref{eq:DMCP} implies that
  $\| z^m \|_1 \leq \|\tilde{z}^m \|_1 = \hop_2(C^+_m,C^-_m)$, completing the proof.
\end{proof}

Later on it will be convenient to assume that each cut $z^m$ is made
up of at most two {\em straight cuts}: straight cuts are defined to be
1-forms $z:\Bonds\to\{-1,0,1\}$ for which there exists a line
$L:=\{x^C+ta_i\sep t\in\R\}$, where $x^C$ is the barycentre of some
$C\in\Cells$ and $a_i$ a nearest neighbour direction, such that
$z_b\neq0$ if and only if the bond satisfies $b\cap L\neq\emptyset$,
\begin{equation*}
  {\rm clos}\b\{x\in\R^2\bsep x\in b \in \Bonds, z_b\neq0\b\}
\end{equation*}
is a connected set, and $z_b>0$ either exclusively on bonds in
the directions $a_{i+1}$ and $a_{i+2}$ or in the directions $a_{i-1}$
and $a_{i-2}$. We will say that a straight cut `lies in the direction
$a_j$' whenever $a_i=\pm a_j$ in the definition of the corresponding $L$.
See the cut depicted in the centre of Figure~\ref{fig:typ_conf} for
a visualisation of the definition.

In the next lemma, we show that we can always choose the decomposition
\eqref{eq:z_decomp} such that each $z^m$ is composed of at most 2 straight
cuts.  We will refer to any $u$ as in the conclusion of Lemma
\ref{th:straight_cuts} as satisfying the {\em straight cuts property}.

\begin{lemma}
  \label{th:straight_cuts}
  Let $u\in\Hsi$, and $Du=\beta+z$ as in \eqref{eq:defn_Du_beta_z}.
  Then there exists $\tilde{u}\in\Hsi$ satisfying the
  \DMCP~\eqref{eq:DMCP} as well as
  $D\tilde{u}=\beta+\sum_{m=1}^{\#\Coresp[\beta]} z^m$ where each
  $z^m$ is the sum of at most 2 straight cuts.
\end{lemma}
\begin{proof}
  The idea is to show that we may always find a shortest path in the
  dual lattice between any pair of cells which is made up of 2
  straight segments. It is intuitively clear from Figure
  \ref{fig:typ_conf} that this can always be done. 

  A complete proof is postponed until Appendix \ref{app:straight_cuts}.
\end{proof}

\subsection{Shifting the Origin}
\label{sec:orig.shift}
Suppose that $y = \yh + u$, where $u \in \Hsi$ satisfies the \DMCP~\eqref{eq:DMCP}.
%
%The decomposition \eqref{eq:z_decomp} of $z$ encodes a link between
%cores of opposite sign and hence allows us to identify dipoles.
%Furthermore, we there is one dislocation core `left over' which is
%contained in $\alh$. This `geometrically necessary dislocation' is the
%essential geometric singularity that we wish to analyse. We now
%exploit additional freedom in making this distinction between the
%`geometrically necessary dislocation' in $\alh$ and the `statistically
%stored' ones contained in $\beta$, which arises due to the fact that
%$\E$ is invariant under the symmetry group of the lattice.  This
%invariance allows us to choose any positive core in $\beta$ as the
%necessary core. We will choose a geometrically necessary core among
%$\Coresp[\alpha]$ that is `well separated' from the dipoles and
%corresponding branchcuts, and then shift this core to the origin. This
%process will result in a stronger version of \eqref{eq:DMCP}
%that will allow us to establish a series of estimates on $\E(u)$.
%
For any $C\in\Cells$, we define the affine transformation
\begin{equation}
  F^C\xi := \cases{
    \xi + x^C & \text{ if the triangle $C$ points upwards, } \\
    \mR_6(\xi+x^C) & \text{ if the triangle $C$ points downwards,}
  }
  \label{eq:FC_defn}
\end{equation}
where $\mR_6$ denotes the rotation through angle $\pi/3$. Thus, $F^C$
maps the lattice onto the lattice, and $C_0$ onto the cell $C$.
% ; the rotation is either present or absent
% depending on whether the cell $C$ is a triangle pointing upwards or
% downwards.

Next, let
\begin{equation}
  u^C := u\circ F^C\,+\,\yh\circ F^C\,-\,\yh.
  \label{eq:uC_defn}
\end{equation}
It follows that $\yh(\xi)+u^C(\xi)=\yh(F^C\xi)+u(F^C\xi)$ for all
lattice points $\xi\in\L$, so that there are corresponding bond length
1-forms $\alpha^C\in[D\yh+Du^C]$ satisfying 
\begin{equation*}
  \alpha^C=\alpha\circ F^C.
\end{equation*} 
As before, define $\beta^C := \alpha^C - \alh$.  According to
these definitions,
\begin{align*}
  \yh(F^C\xi)-\yh(\xi):=\smfrac1{2\pi}\b(\arg(\xi+x^C)-\arg(\xi)\b);
\end{align*}
if we make this function single--valued by introducing a compact polygonal
branch cut passing through the barycentres
of a shortest dual lattice path between $C_0$ and $C$, then it is a
straightforward exercise to show that
\begin{align*}
  |D(\yh\circ F^C)_b - D\yh_{b}| \lesssim \d{b}^{-2};
\end{align*}
therefore $u^C\in\Hsi$, and
\begin{align*}
  \E(u^C)&=E\b(\yh+u^C;\yh\b)\\
   &= E\b(\yh\circ F^C+u\circ F^C;\yh\b)\\
   &= E\b(\yh\circ F^C+u\circ F^C;\yh\circ F^C\b)+E\b(\yh\circ F^C;\yh\b)\\
   &= \E(u),
\end{align*}
noting that the the first term on the third line is simply a
resummation of $\E(u)$, and the second term vanishes. 
% It follows by
% composition of further transformations that $\E$ is invariant under
% lattice symmetries.

For each $C \in \Coresp[\alpha]$, we can replace $u^C$ with
$\tilde{u}^C = u^C + U$ for some $U : \L \to \Z$, such that $\| D u^C
+ D U - \beta^C \|_1$ is minimal, i.e. $\tilde{u}^C$ satisfies the
\DMCP~\eqref{eq:DMCP}. We obtain that
\begin{displaymath}
 \E(\tilde{u}^C) = \E(u^C) =
\E(u). 
\end{displaymath}

To summarize, we have constructed a corrector displacement
$\tilde{u}^C \in \Hsi$ with the same energy as $u$, but for which $C
\in \Coresp[\alpha]$ has been shifted to the origin. Upon minimising
$\| D\tilde{u}^C - \beta^C \|_1$ amongst all choices $C \in
\Coresp[\alpha]$, we obtain the following result.

\begin{lemma}
  Let $v \in \Hsi$, then there exists $u \in \Hsi$ such that $\E(u) =
  \E(v)$ and such that the {\em discrete optimal connection property}
  holds:

  {\bf (\DOCP)} There exists $\alpha \in [D(\yh+u)]$ such that
  \begin{equation}
    \label{eq:DOCP}
    \| D u - \beta \|_1 = \min_{C \in \Coresp[\alpha]} \min_{U : \L \to
      \Z} \| Du^C + D U - \beta^C \|_1,
  \end{equation}
  where $\beta = \alpha - \alh$, $u^C$ is defined by
  \eqref{eq:uC_defn} and $\beta^C=\alpha\circ F^C-\alh$, where $F^C$
  is defined in~\eqref{eq:FC_defn}.
\end{lemma}

\medskip

The crucial property that we obtain from the \DOCP~\eqref{eq:DOCP} is a
bound on the distance between the necessary core at $C_0$ and all negative cores.

\begin{lemma}
  \label{th:wk_hop_min}
  Suppose $u\in\Hsi$ satisfies the \DOCP~\eqref{eq:DOCP} and let $z =
  \sum_{m = 1}^{M} z^m$ according to~\eqref{eq:z_decomp}. Then, 
  \begin{equation}
    \hop_2(C_0,C^-_m) \geq \hop_2(C^+_m,C^-_m), \qquad \text{for } m =
    1, \dots, M, \label{eq:wk_hop_minimality}
  \end{equation}
  where we recall that $z^m$ connects the cores
  $C^+_m\in\Coresp[\alpha]$ and $C^+_m\in\Coresm[\alpha]$.
\end{lemma}

\begin{proof}
  Suppose the converse for contradiction. Then there exists $m$ and a
  dual lattice path connecting $C_0$ to $C^-_m$ which is strictly shorter
  than $\hop_2(C^+_m,C^-_m)$.  Letting $F := F^{C_m^+}$ and $v := u^{C_m^+}$,
  \begin{equation*}
    Dv-\beta\circ F = z\circ F+\tilde{z} = \sum_{m} z^{m}\circ F + \tilde{z},
  \end{equation*}
  where $\tilde{z}$ is the contribution coming from the branch cut in
  $\yh\circ F-\yh$.  Consider the closed curve passing from
  $FC^+_m=C_0$ to $FC_0$ along the branch cut, then along a shortest
  lattice path between $FC_0$ and $FC^-_m$, and then back to $FC^+_m$
  along the support of $z^m\circ F$. This is a closed curve, and by a
  similar argument to that in Corollary \ref{th:short_dual_latt}, we
  can define $w\in\Usz$ as $w(\xi)=1$ for $\xi\in\L$ inside the
  curve, and $0$ outside.  It can then be checked that $v+w$ has a
  corresponding $\bar{z}$ which satisfies
  \begin{equation*}
    \|\bar{z}\|_1 = \|z\circ F\|-\|z^m\circ F\|+\hop_2(C_0,C^-_m)=
      \sum_{i\neq m} \|z^i\|_1+\hop_2(C_0,C^-_m)<\|z\|_1,
  \end{equation*}
  the required contradiction.
\end{proof}

As a corollary we obtain the following stronger property.

\begin{corollary}
  \label{th:str_hop_min}
  Suppose $u\in\Hsi$ satisfies the \DOCP~\eqref{eq:DOCP} and let $z =
  \sum_{m = 1}^{M} z^m$ according to~\eqref{eq:z_decomp}. Then, for
  any $m \in \{1,\dots,M\}$ and for any cell $C \in \Cells$ such that
  $z^m_b\neq0$ for some $b \in \partial C$,
  \begin{equation}
    \hop_2(C_0,C)\geq \hop_2(C^+_m,C). \label{eq:str_hop_min}
  \end{equation}
\end{corollary}

\begin{proof}
  Lemma \ref{th:wk_hop_min} states that, if $C^-\in\Coresm$ and
  $C^+\in\Coresp$ are connected by $z^m$, then
  \begin{equation*}
    \hop_2(C_0,C^-)\geq \hop_2(C^+,C^-).
  \end{equation*}
  It is clear that any subpath of a shortest path in a graph is also a shortest
  path. By the construction of $z^m$, $C$ lies on a shortest path between $C^+$
  and $C^-$, and therefore
  \begin{equation*}
     \hop_2(C^+,C^-)=\hop_2(C^+,C)+\hop_2(C,C^-).
  \end{equation*}
  The triangle inequality for paths \eqref{eq:hop_tri_ineq} now directly implies
  \begin{equation*}
    \hop_2(C_0,C)\geq \hop_2(C^+,C). \qedhere
  \end{equation*}
\end{proof}

\subsection{Estimating \texorpdfstring{$\<\del \E(0),u\>$}{<dE(0),u>}}

In \S\ref{sec:dips_branch} and \S\ref{sec:orig.shift}, we showed that
for any $u\in\Hsi$, we can find $\tilde{u}\in\Hsi$ such that
$\E(u)=\E(\tilde{u})$, and for which the corresponding branch cuts $z$
satisfy the \DOCP \eqref{eq:DOCP}.  We are now in a
position to exploit the chosen structure of $z$ to derive
compactness for minimising sequences.

Our first step is to provide a stronger bound on
$\del\E(0)$. We have already shown that $|\< \del\E(0), u \>|
\lesssim \| D u \|_{\ell^2}$, but this will not be sufficient
  since our estimates so far only provide a bound
  on $\beta$, and not on $Du$ itself. Therefore
we need to estimate $|\< \del\E(0), u \>|$ only in terms of $\|
\beta\|_2$.

\begin{lemma} \label{th:bnd_delE_beta}
  For each $u \in \Hsi$, let $\beta_u := \beta$ be defined
  through~\eqref{eq:defn_beta}.

  There exists a constant $C > 0$ such that
  \begin{displaymath}
    \label{eq:Frcs_bnd}
    \b\< \del\E(0), u \b\> \leq C \|  \beta_u \|_2 \qquad \forall u
    \in \Hsi \quad \text{satisfying the \DOCP~\eqref{eq:DOCP}.}
  \end{displaymath}
\end{lemma}

We provide the proof of this fundamental estimate throughout the
remainder of this section.

Recall the definition of $\xi_0$ from \S\ref{sec:bonds}.
For any $\xi\in\L$, it is always
possible to express the difference $\xi-\xi_0$ as
\begin{equation*}
  \xi-\xi_0 = na_i+ma_{i+1},
\end{equation*}
for some nearest neighbour lattice direction $a_i$ and some $n,m\in\N\cup\{0\}$ with $n\neq0$
unless $\xi=\xi_0$. We then define the path $\Gamma_\xi$ to be
\begin{equation*}
  \Gamma_\xi:=\sum_{j=0}^{n-1} \b(\xi_0+ja_i,\xi_0+(j+1)a_i\b)+\sum_{j=0}^{m-1}
    \b(\xi_0+na_i+ja_{i+1},\xi_0+na_i+(j+1)a_{i+1}\b);
\end{equation*}
cf. Figure \ref{fig:gammaxi}.  Integrating $Du_b$ along $\Gamma_\xi$,
we obtain
\begin{align}
  |u(\xi)| &= \bg|\int_{\Gamma_\xi}Du\,\bg|=\bg|\int_{\Gamma_\xi}\beta+z\bg|\notag\\
    &\leq |\Gamma_\xi|^{1/2}\bg(\sum_{b\in\Gamma_\xi}|\beta_b|^2\bg)^{1/2}
    + \Bg| \int_{\Gamma_\xi} z\,\Bg|\notag\\
    &\lesssim |\xi|^{1/2} \|\beta\|_2 + \Bg| \int_{\Gamma_\xi} z\,\Bg|,
  \label{eq:u_est}
\end{align}
using the Cauchy-Schwarz inequality. We now bound the final term in
\eqref{eq:u_est}.

\begin{lemma}
\label{lem:z_crossings}
Suppose $u\in\Hsi$ satisfies the \DOCP~\eqref{eq:DOCP} and the
straight cuts property (cf. Lemma \ref{th:straight_cuts}); then
\begin{equation*}
  \bg|\int_{\Gamma_\xi} z \,\bg|\lesssim \min\b\{|\xi|^2,
  \#\Coresp[\alpha] \b\}.
\end{equation*}
\end{lemma}

\begin{proof}
First, we note that for any straight cut $z'$,
\begin{equation*}
  \bg|\int_{\Gamma_\xi} z' \bg|\leq1.
\end{equation*}
This follows from the fact that all $b\in\Bonds$ for which $z'_b=1$ can
be written as
\begin{equation*}
  (\xi+na_j,\xi+na_j+a_{j+1})\quad\text{or}\quad(\xi+na_j,\xi+na_j+a_{j+2})
\end{equation*}
for some $\xi\in\L$, $n\in\N$ and some nearest neighbour lattice direction
$a_j$, and the definition of $\Gamma_\xi$.
Using the straight cuts property, we find that
\begin{equation*}
  \bg|\int_{\Gamma_\xi} z
  \bg|\leq\sum_m\bg|\int_{\Gamma_\xi}z^m\bg|\leq2 \, \#\Coresp[\alpha].
\end{equation*}

To prove the second bound, enumerate the bonds $b_j\in\Gamma_\xi$,
beginning at with the bond $b_1=(\xi_0,\xi_0+a_i)$. Each
$b_j\in\partial C^{b_j}$ for some $C^{b_j}\in\Cells$. Since $u$
satisfies the \DOCP, Corollary~\ref{th:str_hop_min}, $z^m_{b_j}\neq0$
implies that if $z^m$ connects to $C^+\in\Coresp[\alpha]$, then
\begin{equation*}
  \hop_2(C^+,C^{b_j})\leq \hop_2(C_0,C^{b_j})\leq\hop_2(C_0,C^{b_0})+2j,
\end{equation*}
where we have repeatedly applied the triangle inequality
\eqref{eq:hop_tri_ineq}. Further application of the triangle
inequality implies that the sequence of sets
\begin{equation*}
  \b\{C\in\Cells\bsep \hop_2(C,C^{b_j})\leq \hop_2(C_0,C^{b_0})+2j\b\}
\end{equation*}
for $j=1,\ldots,|\Gamma_\xi|$ is increasing, and it is easy to see that
\begin{equation*}
  \#\b\{C\in\Cells\bsep \hop_2(C,C^{b_j})\leq \hop_2(C_0,C^{b_0})+2j\b\}
    \lesssim j^2.
\end{equation*}
Since by \eqref{eq:bvec_cell_leq_1} each cell can contain
at most one dislocation core, we must therefore have
\begin{equation*}
  \bg|\int_{\Gamma_\xi} z\bg| \lesssim |\Gamma_\xi|^2 \lesssim |\xi|^2. \qedhere
\end{equation*}
\end{proof}

\begin{figure}
  \includegraphics[height=6cm]{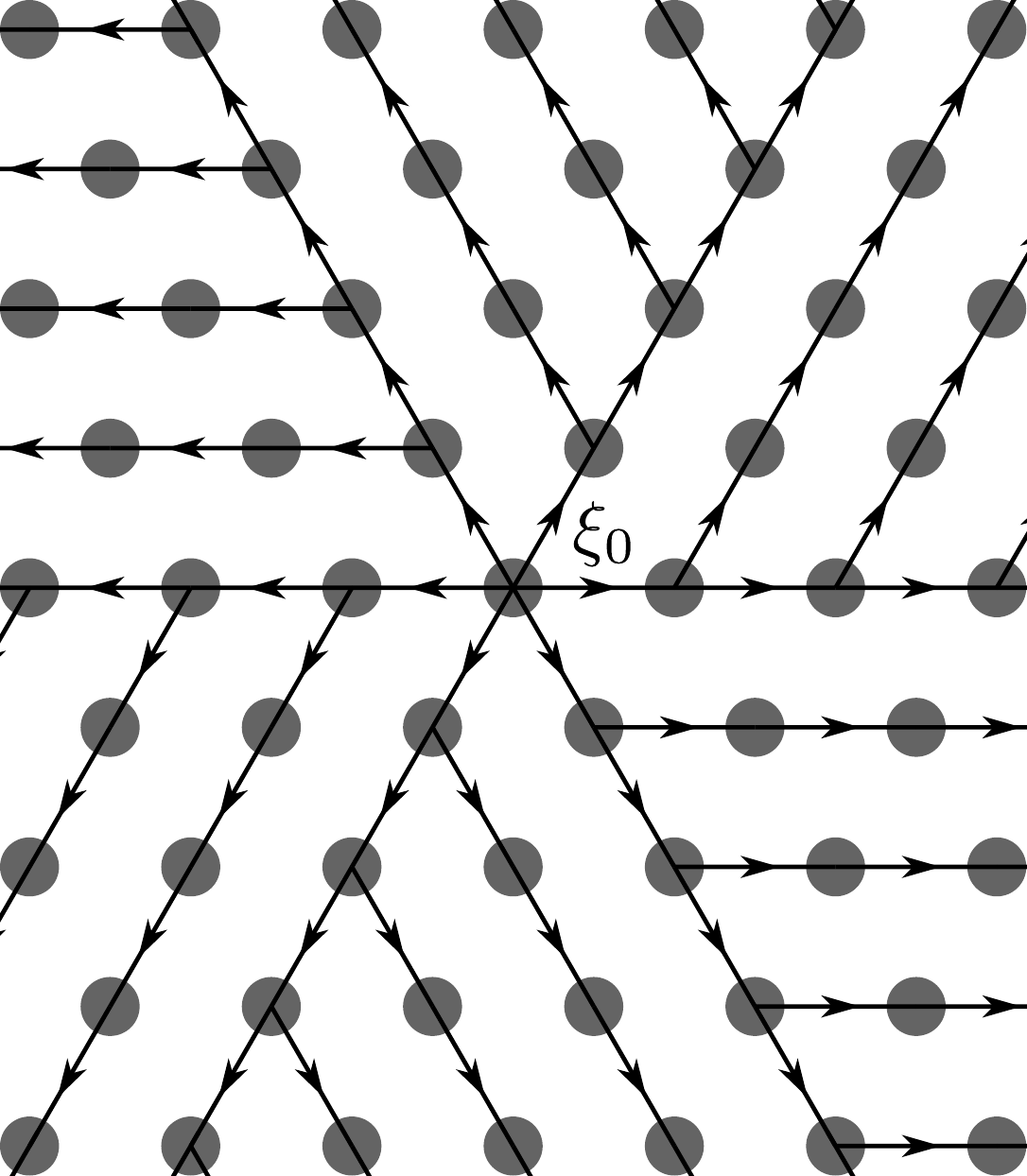}
  \caption{An illustration of the definition of $\Gamma_\xi$.}
  \label{fig:gammaxi}
\end{figure}

Next, we show that the number of cores can be bounded in terms of
$\beta$. This is intuitive, since, as discussed in \cite{Ponsiglione07},
each core stores a positive amount of elastic energy. Let $C \in \Cores[\beta]$,
then Jensen's inequality implies
\begin{equation}
  \label{eq:bnd_single_core_beta}
  \int_{\partial C} |\beta|^2 \geq \frac13 \bg|\int_{\partial C}\beta\bg|^2 = \frac13.    
\end{equation}
Hence, we obtain 
\begin{equation}
  \label{eq:bnd_cores_beta}
  \#\Cores[\beta] \leq 3 \| \beta \|_2^2.
\end{equation}

\medskip

\begin{proof}[Proof of Lemma \ref{th:bnd_delE_beta}]
  We now combine Lemma \ref{lem:z_crossings}, Lemma
  \ref{th:decay_frc_yh} and \eqref{eq:bnd_cores_beta} to estimate
  \begin{align*}
    |\<\del\E(0),u\>| &\leq \sum_{\xi\in\L} |f(\xi)| |u(\xi)|,\\
    &\lesssim \sum_{\xi\in\L} \B\{ \|\beta\|_2 |\xi|^{-5/2} +\min\b(\|\beta\|_2^2,|\xi|^2\b)|\xi|^{-3}\B\}.
  \end{align*}
  We note that the $|\xi|^{-5/2}\in\ell^1(\L)$, so that the first term
  is bounded above by $C\|\beta\|_2$. The second term splits into
  \begin{equation*}
    \sum_{\xi\in\L}\min\b(\|\beta\|_2^2,|\xi|^2\b)|\xi|^{-3}\leq \sum_{|\xi|\leq \|\beta\|_2}|\xi|^{-1}+
    \|\beta\|_2^2 \sum_{|\xi|> \|\beta\|_2} |\xi|^{-3}.
  \end{equation*}
  Straightforward radial estimates yield the bounds
  \begin{displaymath}
    \sum_{|\xi|\leq \|\beta\|_2}|\xi|^{-1}\lesssim \|\beta\|_2 \qquad
    \text{and} \qquad
    \sum_{|\xi|> \|\beta\|_2} |\xi|^{-3}\lesssim \|\beta\|_2^{-1}.
  \end{displaymath}
  Combining the previous estimates, we obtain the stated result.
\end{proof}

\subsection{Estimating \texorpdfstring{$\sum_{b\in\Bonds}z_b\psi'(\alh_b)$}{cut contributions}}
\label{sec:z_psi_est}
In this section, we estimate the second group in
\eqref{eq:lb:0}. Unlike in our previous estimates, which are generic,
we now resort to precise quantitative bounds based heavily on our
  assumption $(\psi 5)$.

\begin{lemma}
\label{lem:cut.lwr.bnd}
Suppose that $u\in\Hsi$ satisfies the \DOCP~\eqref{eq:DOCP}, then for any cut $z^m$
connecting to a dipole $C^+\in\Coresp[\alpha]$ to
$C^-\in\Coresm[\alpha]$, we have the lower bound:
\begin{equation}
  \sum_{b\in\Bonds}z^m_b\psi^\prime(\alh_b)\geq -\psi^{\prime\prime}(0)
    \frac{\arcsinh\b(2/\sqrt{3}\b)}{\pi} - c_0\,\hop_2(C_0,C^-)^{-1},
    \label{eq:zm_psi_bnd}
\end{equation}
where $c_0 > 0$ is independent of $u$.
\end{lemma}

\begin{remark}
  We have no reason to believe that the estimate \eqref{eq:zm_psi_bnd}
  is sharp. It is sufficient for our purpose, due only to the fairly
  strong technical assumption $(\psi5)$ made in
  \S\ref{sec:Ediff_defn}.
\end{remark}

\medskip

The complete proof of Lemma \ref{lem:cut.lwr.bnd} is given in Appendix
\ref{app:cut_proofs}, but since this is a crucial part of our
analysis, we provide a brief sketch.

\begin{proof}[Sketch of the proof of Lemma \ref{lem:cut.lwr.bnd}]
  A crucial consequence of the \DOCP is \eqref{eq:str_hop_min}, since
  this inequality says that the dipoles described by each $z^m$
  satisfy
  \begin{align*}
    \text{either: } \quad & \hop_2(C_0,C^+)\leq\hop_2(C_0,C^-); \\
    \text{or: } \quad & \hop_2(C_0,C^-)<\hop_2(C_0,C^+), \quad
    \text{but} \quad \hop_2(C^+,C^-)\leq\hop_2(C_0,C^-).
  \end{align*}

  Since there is a positive dislocation core present in $C_0$, dipoles
  in the first category give a positive contribution to the sum, since
  the repulsive force between $C^+$ and $C_0$ dominates.  

  In the second case, the attractive forces between $C_0$ and $C^-$
  dominate, hence these dipoles give a negative contribution to the
  energy.  Requiring the straight cuts property derived in
  Lemma~\ref{th:straight_cuts} allows us to obtain estimates on the
  terms in the sums
  \begin{equation*}
    \sum_{b\in\Bonds}\psi'(\alh_b)z^m_b
  \end{equation*}
  using \eqref{eq:alpha.hat.bnd}, and estimating sums in terms of
  integrals that can be evaluated explicitly. Thus, we obtain explicit
  bounds for various different cases in terms of $\hop_2(C_0,C^-)$ and
  $\hop_2(C^+,C^-)$.  These estimates are logarithmic, and hence turn
  out to essentially depend upon the ratio
  \begin{equation*}
    \frac{\hop_2(C^+,C^-)}{\hop_2(C_0,C^-)} \leq 1;
  \end{equation*}
  the final result is the bound stated in \eqref{eq:zm_psi_bnd}.
\end{proof}

%The proof of this lemma is relatively technical, but not particularly
%enlightening, so we postpone it to Appendix \ref{app:cut_proofs}.
%To construct this estimate, we use the decomposition into $z^m$ made up of at
%most 2 straight cuts. These straight cuts can further be categorised into 2
%kinds: tangential and radial cuts. The essential idea is to show that dipoles
%only give a negative contribution to the energy if the negative core is closest
%to the geometrically necessary positive core, which lies at the origin. The
%definition of $z$ as having minimal $\ell^1$ norm implies that such cuts have
%length which is no greater than the distance from the negative core to the 
%central core. This sufficiently limits the size of the negative contribution
%from such dipoles to allow us to obtain the bound.

We have now collected all estimates required to obtain a coercivity
result. Although we state the result for general $u \in \Hsi$, we will
only require it later for $u$ satisfying the \DOCP.

\begin{theorem} 
  \label{th:lb:1}
  Let $u\in\Hsi$, then there exists $\alpha\in[D\yh+Du]$ and
  $\beta=\alpha-\alh$, such that
  \begin{equation}
    \E(u)\geq c_1\,\|\beta\|_2^2 + c_2\,\#\Cores[\beta] - c_3,
    \label{eq:lb:1}
  \end{equation}
  where $c_i > 0$ are independent of $u$.
\end{theorem}
\medskip

\begin{remark}
We note that since we will show $\|\beta\|_2^2\gtrsim\#\Cores[\beta]$, we
could write \eqref{eq:lb:1} more concisely as
\begin{equation*}
  \E(u)\geq c_1\,\|\beta\|_2^2 - c_3. \qedhere
\end{equation*}
\end{remark}

\begin{proof}
  The bound is clearly invariant under the vertical shift and
  horizontal shift and rotation transformations we applied in
  \S\ref{sec:dips_branch} and \S\ref{sec:orig.shift}. Without
  loss of generality, we may therefore assume that $u$ satisfies the
  \DOCP~\eqref{eq:DOCP}. Moreover, according to Remark
  \ref{rem:alpha_ambig} we can choose $\alpha \in [D\yh+Du]$ in such a
  way that $\alpha_b = 0$ on {\em any} bond that lies on the
  intersection between two cores.

  Summing \eqref{eq:zm_psi_bnd} over $m$, we find that
  \begin{equation}
    \sum_{b\in\Bonds}z_b\psi'(\alh_b)\geq -\psi^{\prime\prime}(0)
    \frac{\arcsinh(2/\sqrt{3})}{2\pi}\#\Cores[\beta]
    -c_0\sum_{C\in\Coresm[\beta]}\hop_2(C,C_0)^{-1}.
    \label{eq:z_psi_bnd}
  \end{equation}
  for some $c_0>0$. 

The second group in \eqref{eq:z_psi_bnd} can be estimated using the
fact that each cell can contain no more that 1 dislocation core, so
that for any $\delta > 0$, there is a constant $C_\delta$ such that
\begin{equation*}
  \sum_{C\in\Coresm[\beta]}\hop_2(C,C_0)^{-1}\geq -\del\,\#\Cores[\beta]-C_\del.
\end{equation*}
Bringing together \eqref{eq:lb:0}, \eqref{eq:Frcs_bnd} and \eqref{eq:z_psi_bnd},
we have that for arbitrary $\eps>0$ and $\del>0$
\begin{equation*}
  \E(u)\geq \b(\smfrac12\psi^{\prime\prime}(0)-\eps)\|\beta\|_2^2
   -\bg(\psi^{\prime\prime}(0)\frac{\arcsinh(2/\sqrt3)}{2\pi}+\del\bg)\#\Cores[\alpha]
   -C_{\eps,\delta}.
\end{equation*}

Since we assumed that $\alpha_b = 0$ on any bond that is adjacent to
two cores (cf. Remark \ref{rem:alpha_ambig}), we obtain from
\eqref{eq:bnd_single_core_beta} that
\begin{displaymath}
  \|\beta\|_2^2 \geq \smfrac13 \#\Cores[\beta].
\end{displaymath}
Since
\begin{equation*}
  \frac13 > \frac{\arcsinh(2/\sqrt{3})}{\pi} \approx 0.314,
\end{equation*}
the result follows by taking $\eps$ and $\del$ small enough.
\end{proof}

\subsection{Existence of Minimisers of \texorpdfstring{$\E$}{E}}
\label{sec:existence}
With the coercivity result of Theorem \ref{th:lb:1} in place, we are
now in a position to apply the Direct Method and establish existence
of a minimiser of $\E$ in $\Hsi$.

Take a sequence $u^n\in\Hsi$ such that
\begin{equation*}
  \E(u^n)\to \inf_{u\in\Hsi}\E(u).
\end{equation*}
Referring back to \S\ref{sec:orig.shift}, we may assume that $u^n$
satisfies the \DOCP. Let $\alpha^n \in [D\yh+Du^n]$ satisfy the
condition of Theorem \ref{th:lb:1}, and $\beta^n := \alpha^n -
\alh$. Theorem \ref{th:lb:1} then implies that $\beta^n$ has a weakly
convergent subsequence in $\ell^2(\Bonds)$. In the next lemma, we also
obtain convergence of $Du^n$.

\begin{lemma}
  \label{th:coercivity}
  Suppose $u^n\in\Hsi$ is a minimising sequence for which each $u^n$
  satisfies the \DOCP~\eqref{eq:DOCP}. Then there exists a subsequence
  which converges weakly in $\Hsi$, and the corresponding $z^n$
  converges weakly in any $\ell^p(\Bonds)$ with $1<p<2$.
\end{lemma}

\begin{proof}
  Corollary \ref{th:lb:1} implies that, selecting a subsequence of
  $u^n$ (not relabelled), we may assume that $\beta^n \rightharpoonup
  \beta$ weakly in $\ell^2(\Bonds)$ and that $M:=\#\Coresp[\beta^n]$
  is constant along the sequence.

  Let $Du^n = \beta^n + z^n$. Lemma \ref{lem:z_decomp} implies that
  each $z^n$ can be decomposed into
  \begin{equation*}
    z^n=\sum_{m=1}^Mz^{n,m}.
  \end{equation*}
  
  Let $B$ be any finite sum of positively-oriented cells. Since
  weak convergence in $\ell^2(\Bonds)$ implies pointwise convergence,
  it follows that
  \begin{equation*}
    \int_{\partial B} \beta^n = -\int_{\partial B} z^n \to N \in\Z
    \qquad \text{ as } n \to \infty.
  \end{equation*}

  % Let $A := \sum \Cores[\beta]$ and let $N := \int_{\partial A}
  % \beta$. We claim that $N = 0$.  Then, for any finite sum of cells
  % $B$, $B \supset A$, we have
  % \begin{equation*}
  %   \lim_{n\to\infty}\int_{\partial B} \beta^n =
  %   \lim_{n\to\infty}\int_{\partial A} \beta^n = N.
  % \end{equation*}
  % This follows from the fact that, 

  We enumerate the cores $C^{n,m} \in \Cores[\alpha^n]$, $n \in \N$,
  $m = 1, \dots, 2M+1$. Let $\mathcal{M}_{\rm bdd}$ be the set of
  indices of cores that remain at a bounded distance from the origin,
  that is,
  \begin{displaymath}
    \mathcal{M}_{\rm bdd} := \B\{ m \in \{ 1,\dots,2M+1 \} \Bsep
    \sup_{n \in \N} \d{C^{n,m}} < +\infty  \B\}.
  \end{displaymath}
  Since the core centres $x^{C^{n,m}}$ with $m \in \mathcal{M}_{\rm
    bdd}$ can only take a finite number of positions, we can extract
  a further subsequence (not relabelled) so that they are constant.

  We therefore observe that
  \begin{displaymath}
    A := \sum\Cores[\beta] = \sum_{m \in \mathcal{M}_{\rm bdd}}
    C^{n,m}, \quad \text{ for all $n$.}
  \end{displaymath}

 Then, for any finite sum of cells $B$, $B
  \supset A$, we have
  \begin{equation*}
    \lim_{n\to\infty}\int_{\partial B} \beta^n =
    \lim_{n\to\infty}\int_{\partial A} \beta^n = \int_{\partial A}
    \beta =: N \in \Z.
  \end{equation*}
  
  We aim to show that $N = 0$. For all $n$ sufficiently large,
  applying Jensen's inequality implies
  \begin{equation*}
    \int_{\partial A} |\beta^n|^2 \geq \frac1{|\partial A|}
    \bg|\int_{\partial A} \beta^n \bg|^2=\frac{N}{|\partial A|}.
  \end{equation*}
  Let $B^0_0 \supset A$ be a finite sum of cells that form a convex
  lattice polygon, and let
  \begin{equation*}
    B^0_k=B^0_{k-1}\cup\{C:\overline{\partial C}\cap B^0_{k-1}\neq\emptyset\},
  \end{equation*}
  for $k \in \N$.  By considering all possible corners for a convex
  lattice polygon it is straightforward to show that $|\partial
  B^0_{k}|=|\partial B^0_{k-1}|+6$. Since $\lim_{n\to\infty}
  \int_{\partial B^0_k} \beta^n = N \in \Z$, there exist $n_k$ such
  that $\int_{\partial B^0_i} \beta^{n_k} = N$ for $i = 1, \dots, k$,
  and hence,
  \begin{equation*}
    \|\beta^{n_k}\|_2^2 \geq \sum_{i=0}^k\int_{\partial B^0_i}|\beta^{n_k}|^2
    \geq \sum_{i=0}^k \frac{N}{|\partial B^0_k|}
    \gtrsim N\log(k).
  \end{equation*}
  Since $\|\beta^{n_k}\|_2$ is bounded, we obtain that $N=0$. We can
  therefore conclude that, for any finite sum of positively oriented
  cells $b$, $B \supset A$,
  \begin{displaymath}
    \int_{\partial B}z^n \to 0 \quad \text{as}  \quad n\to\infty.
  \end{displaymath}

  Using a concentration compactness argument, we now show that we can
  ``group'' those cores which diverge into sums of cells with net
  Burgers vector zero.

  To that end, define the lattice translation operators
  $F^n_m:=F^{C^n_m}$ as in \eqref{eq:FC_defn}, where
  $C^n_m\in\Coresp[\beta^n]$ and $\beta^n_m:=\beta^n\circ F^n_m$. Note
  that
  \begin{equation*}
    \b\|\beta^n_m\b\|_2=\b\|\beta^n\b\|_2,
  \end{equation*}
  so $\beta^n_m$ is a bounded sequence for each $m$, and we can select
  a subsequence such that $\beta^n_m\wto \beta_m$ for some $\beta_m$
  in $\ell^2(\Bonds)$ and for each $m=1,\ldots,M$.  As above, it
  follows that there exists a finite sum of positively-oriented cells
  $A^m$ which contains $C_0$ and all dislocation cores of $\beta_m$,
  and is such that for any sum of positively-oriented cells $B$
  containing $A^m$ as a subsum,
  \begin{equation}
    \label{eq:0_burgers_loc}
    -\int_{\partial B} z^n \circ F^n_m = \int_{\partial B} \beta^n_m
    = 0   \quad \text{for $n$ sufficiently large.}
  \end{equation}

  Let $A^n_m:=(F^n_m)^{-1}A^m$, then we have shown that, for all $n$
  sufficiently large, all dislocation cores in $\beta^n$ lie within
  the set
  \begin{equation*}
    S^n:=\bigcup_{m=1}^M A^n_m.
  \end{equation*}

  We are now ready to establish that $\|z^n\|_1$ is bounded. Since
  each $z^n$ satisfies the \DMCP, $\|z^n\|_\infty \leq M$.  We
  therefore simply need to rule out the possibility that
  $\#\supp\{z^n\}\to\infty$ as $n\to\infty$ (i.e., that the branch cut
  lengths diverge). 

  Fix some $m \in \{1, \dots, M\}$ and suppose that the core $C^n_m
  \in \Coresp[\beta^n]$ is connected to $K^n_1 \in \Coresm[\beta^n]$
  via a ``cut'' $z^{n,m}$. We claim that $K^n_1 \in A^{n,m}$.

  If this were false and $K^n_1 \in A^n_{\ell_1}$ where $\ell_1 \neq
  m$, then there must be $L^n_1 \in \Coresp[\beta^n]$, $L^n_1 \in
  A^{n}_{\ell_1}$, which is connected to another core $K^n_2 \in
  \Coresm[\beta^n]$ outside of $A^n_{\ell_1}$ and outside
  $A^n_m$. Upon iterating this construction, we find a series of cores
  $K^n_1, L^n_1, K^n_2, L^n_2, \dots$, which must eventually
  repeat. 

  Let $L^n_1 := C^n_m$. We know, by construction of the groups
  $A^n_{\ell}$ that $\hop_2(L^n_i,K^n_{i+1}) \to \infty$ but
  $\hop_2(K^n_i,L^n_{i})$ is bounded as $n\to\infty$ for each
  $i$. This clearly contradicts the \DMCP~and hence the \DOCP.

  Hence, the claim that $K^n_1 \in A^{n,m}$ follows, and this
  immediately implies that $\|z^n\|_1$ is bounded.

  % Note that if $F^n_mC^n_l$ remains bounded for all $n$, then it is
  % easy to show that $A^n_m=A^n_l$ for $n$ sufficiently large. Suppose
  % now that $C^n_m$ is connected to $K^n_1\in\Coresm[\beta^n]$ by
  % $z^{n,m}$. We claim that $F^n_mK^n_1\in B^m$.

  % If not, then $F^n_mK^n_1\to\infty$ as $n\to\infty$, and $K^n_1\in
  % A^n_l$ for some $l$.  But then \eqref{eq:0_burgers_loc} implies
  % there must be $K^n_2\in \Coresp[\beta^n]$ such that $F^n_lK^n_2\in
  % B^l$ and $K^n_2$ is connected to another core
  % $K^n_3\in\Coresm[\beta^n]$ outside $A^n_l$; iterating this
  % procedure, we find a series of cores with Burgers vector of
  % alternating sign which must eventually repeat, since there are only
  % finitely many cores in $\beta^n$.  Furthermore, we know that
  % $\hop_2(K^n_{2i},K^n_{2i+1})\to\infty$ but
  % $\hop_2(K^n_{2i-1},K^n_{2i})$ is bounded as $n\to\infty$ for each
  % $i$. It is now clear that this is impossible, since we know that
  % $z^n$ satisfies the \DOCP.  The claim follows, and we have that
  % $\supp\{z^n\}\subseteq \mathrm{clos}(S^n)$, which is of bounded
  % measure, and hence $\|z^n\|_1$ is bounded.

  Since $\|z_n\|_1$ is bounded and $\ell^1$ compactly embeds into $\ell^p$
  for any $p>1$, it follows that we may extract a further subsequence
  which weakly converges in some $\ell^p(\Bonds)$ with $1<p<2$. This
  further imples that both $z^n$ and $\beta^n$ converge weakly in
  $\ell^2(\Bonds)$, and in particular that $u^n$ converges weakly in
  $\Hsi$, as required.
\end{proof}

We now complete the proof of our main result.

\begin{proof}[Proof of Theorem \ref{th:minim_E}]
  Invoking Lemma \ref{th:coercivity}, suppose that the minimising
  sequence $u^n \rightharpoonup u$ in $\Hsi$, which is a candidate
  minimiser for $\E$, and furthermore that the corresponding $z^n$
  converges weakly in $\ell^p(\Bonds)$ with $1<p<2$.
  It remains to show that $\liminf\E(u^n) \geq \E(u)$.
  Recall from \eqref{eq:explicit_E_extension} that
  \begin{equation*}
    \E(u)=\sum_{b\in\Bonds} \b[\psi(\alh_b+Du_b)-\psi(\alh_b)-\psi'(\alh_b)Du_b\b]
    +\<\del\E(0),u\>.
  \end{equation*}
  Since the second term is a bounded linear functional, it is weakly
  continuous.
  Using \eqref{eq:alpha.hat.bnd}, it may be shown that the linear functional $L$ defined to be
  \begin{equation*}
    L(z):=\sum_{b\in\Bonds}\psi'(\alh_b)z_b
  \end{equation*}
  is in $\b(\ell^p(\Bonds)\b)^*$ for all $p<2$, and is therefore also weakly continuous
  along the sequence.
  From \eqref{eq:psi_lb} we obtain that there exists $R_0,
  \lambda > 0$ such that, for $\d{b} \geq R_0$,
  \begin{equation*}
    \psi(\alh_b+Du^n_b)-\psi(\alh_b)-\psi'(\alh_b)\beta^n_b \geq \lambda |\beta^n_b|^2 \geq 0.
  \end{equation*}
  We can therefore apply Fatou's lemma to obtain
\begin{equation*}
  \liminf_{n\to\infty}\,\sum_{b\in\Bonds} \psi(\alh_b+Du^n_b)-\psi(\alh_b)
    -\psi'(\alh)\beta^n_b \geq \sum_{b\in\Bonds} \psi(\alh_b+Du_b)-\psi(\alh_b)
    -\psi'(\alh_b)\beta_b,
\end{equation*}
and thus, in combination with the weak continuity of $L$ and $\del\E(0)$ along
the minimising sequence,
\begin{equation*}
  \inf_{v\in\Hsi}\E(v) = \liminf_{n\to\infty}\E(u^n) \geq \E(u) \geq \inf_{v\in\Hsi}\E(v),
\end{equation*}
completing the proof of Theorem \ref{th:minim_E}.
\end{proof}

\section{Conclusion}

We have presented a model which has allowed us to analyse the
stability of screw dislocations under anti-plane deformation, and we
have obtained the surprising result that single screw dislocations
exist as globally stable states, i.e. they are global energy minimisers
among all finite energy displacements.  Further, we then showed
that configurations with arbitrarily many dislocations of arbitrary
sign are locally stable, as long as the dislocation cores are suitably
separated, but such configurations do not appear to be globally
stable.

Our work suggests two immediate directions for future study. Firstly,
our analysis relies crucially on the technical assumption $(\psi5)$
made in \S\ref{sec:Ediff_defn}. It would be of interest to understand
the extent to which this assumption could be weakened, but to do so
would require a more qualitative estimate than the quantitative one
made in Lemma \ref{lem:cut.lwr.bnd}, and would require a deeper
insight into the geometry of dipole interaction. Secondly, it would be
interesting to understand which other lattice defects are globally
stable states; this seems far from
clear to us at present.

\section*{Acknowledgements}

The authors would like to thank Florian Theil for originally posing the
problem of describing the existence of dislocations in atomistic problems,
Adriana Garroni for discussions regarding the literature on dislocations,
and Sylvia Serfaty for highlighting the work contained in \cite{BCL86} to us.

\appendix

\section{Analysis of Branchcuts: Proofs}
\label{app:cut_proofs}

In this appendix, we detail the proofs of various geometrical lemmas from
\S \ref{sec:dips_branch}.

\subsection{Proof of Lemma \ref{lem:z_decomp}}
\label{app:z_decomp}

We will prove this lemma algorithmically. First, if $z_b=0$ everywhere, then
the result is trivial. Next, we note that $\beta$ contains an even number of dislocation
cores, and $\#\Coresp[\beta]=\#\Coresm[\beta]<+\infty$ since $u\in\Hsi$.
We therefore enumerate $C^+_i\in\Coresp[\beta]$.

Put $C_{1,0}=C^+_1$. Since $C^+_1$ is a positive
dislocation core and $z_b$ is integer-valued, it follows that at least
one bond $b_1\in\partial C^+_1$ satisfies $z_b>0$. Let $C_{1,1}$ be
the cell such that $-b_1\in C_{1,1}$. There are now 2 possibilities:
either $C_{1,1}\in\Coresm[\beta]$, in which case we stop, or we can
find another bond $b_2\in\partial C_{1,1}$ such that
$z_b>0$. Iterating, we obtain a (possibly infinite) sequence of cells
$C_{1,j}$ and bonds $b_j$.

We now claim that no two bonds $b_j=b_k$ with $j\neq k$ in this sequence,
and consequently the sequence terminates; suppose the converse for
contradiction. Let $j$ and $k$ be indices such that $j<k$ and $k-j$ is minimal
over all pairs of indices such that $b_j=b_k$.
Define a polygonal curve $P$ passing
through the barycentres of the cells $x^{C_{1,j}},\ldots,x^{C_{1,k}}$. 
$P$ is a simple continuous closed curve, since the cells $C_{1,j},\ldots,C_{1,k-1}$
are distinct by definition, and $C_{1,j}=C_{1,k}$. Hence, $P$ partitions $\R^2$ into
a bounded set $\Omega$ (the interior of $P$) and an unbounded set, $\R^2\setminus \Omega$.
Define $\tilde{u} = u \mp \mathbbm{1}_\Omega$, taking the sign according to whether
$P$ traverses $\partial\Omega$ in an anticlockwise or clockwise direction respectively.

It is now straightforward to check that $D\tilde{u} = Du$ except on the bonds $b_i$.
For each of the bonds $b_i$,
\begin{equation*}
  D\tilde{u}_{b_i} = Du_{b_i} - 1 = \beta_{b_i} + z_{b_i}-1,
\end{equation*}
but since $z_{b_i}\geq1$, this contradicts the \DMCP, \eqref{eq:DMCP}. It follows
that the sequence $b_i$ contains no two identical bonds, and as
$z_b$ has compact support, the sequence must terminate at a negative dislocation core.

Define $z^1_b=\pm1$ if $\pm b\in\{b_1,\ldots,b_n\}$, and $z^1_b=0$ otherwise, and
then consider iterating the procedure described above starting at $C_i^+$, but using the criterion at each step
that each bond in the sequence must satisfy
\begin{equation*}
  z_b-\sum_{m=1}^{i-1} z^m_b>0.
\end{equation*}
This leads to a sequence 1-forms, $z^i$; the resulting 1-form
\begin{displaymath}
  z_b-\sum_{i=1}^N z^i_b
\end{displaymath}
must be identically zero. If not, then the same technique as used above shows that
that either the \DMCP~is violated, or else $\supp\{z\}$ is infinite, and hence
$u\notin\Hsi$.

To complete the proof of Lemma
\ref{lem:z_decomp}, we simply need to show that for 2 adjacent bonds in the
support of any $z^m$, $\partial b_i\cap\partial b_{i+1}$ is a single 0-cell. This
is clear, since by definition $b_i,-b_{i+1}\in\partial C_{m,i}$, $b_i\neq b_{i+1}$
and $|\partial C_{m,i}|=3$.

\subsection{Proof of Lemma \ref{th:straight_cuts}}
\label{app:straight_cuts}
We now show that we can always choose $z^m$ to be made up of 2 straight cuts.
Recall that Corollary \ref{th:short_dual_latt} states that
$\|z^m\|_1=\hop_2(C_m^+,C_m^-)$.

Next, define the 2-cell hop operators $H_i$ for $i=1,\ldots 6$ which act on
2-cells by taking a cell $C$ to the first cell `in the direction $a_i$', that
is, the positively-oriented cell which satisfies $C'\neq C$,
\begin{equation*}
  x^C+\smfrac{\sqrt3}{2}a_i\in C'.
\end{equation*}
It is straightforward to check this is well-defined; see Figure
\ref{fig:hopops}. We can represent paths in the dual lattice by words taken
from the alphabet of operators $\{H_1,\ldots, H_6\}$. In general however, the
representation is non-unique -- to see this, it is clear that for $C_2$ in
Figure~\ref{fig:hopops}, $H_2C_2=H_1C_2$. As with the vectors $a_i$, we define
\begin{equation*}
  H_{i+6m}:=H_i
\end{equation*}
for any $m\in\Z$.

We note the following properties of the operators $H_i$, which may be easily
checked:
\begin{enumerate}
  \item The orbit of the group of operators generated by $\{H_1^2,\ldots,H_6^2\}$
    acting on any cell $C$ is a lattice.
  \item The operators $H_i^2$ and $H_j$ commute for any $i$ and $j$.
  \item For any cell $C$ and any $i$, one of the following is true:
    $H_iC=H_{i+1}C$ or $H_iC=H_{i-1}C$.
  \item If $H_iC=H_jC$, then $H_iH_k^{2m}C=H_jH_k^{2m}C$, and if $H_iC\neq H_jC$, 
    then $H_iH_k^{2m}C\neq H_jH_k^{2m}C$.
  \item $H_{i+3}H_i=H_{i+3}H_i$ is the identity map for any $i$.
  \item $H_iC\neq H_{i\pm2}C$ for any cell $C$ and for any $i$.
\end{enumerate}
Setting $N:=\|z^m\|_1$, we claim that for any pair of cells $C_m^\pm$, it is possible
to write shortest paths in the dual lattice as a word of the form
\begin{equation}
  H_{i+1}^{N-k}H_i^k \label{eq:shortest_path_ops}
\end{equation}
for some $i$ and some $k\in\{1,\ldots,N\}$. By the definition of the hopping
operators, it is clear that such words represent a sequence of cells lying
on the lines $x^{C_m^-}+ta_i$ and $x^{C_m^+}-ta_{i+1}$.

We now prove the claim: first, we show that any shortest path must be able to be
written as a word made up of only 2 of the operators $H_i$ and $H_{i+1}$. Suppose
this is false, for contradiction. (6) implies that we may assume that it contains
a segment which may be written as
\begin{equation*}
  H_{i+2}H_i^mH_{i-2}\quad \text{or}\quad H_{i-2}H_i^mH_{i+2}
\end{equation*}
for some $i$ and some $m\in\N$. Since both cases are similar, we consider only the
first. If $C$ is the first cell in this subsequence, $H_{i-2}C\neq H_{i-3}C$,
or else $H_iH_{i-2}C=H_iH_{i-3}C=C$ by (5), and there exists a shorter path.
(3) therefore implies
\begin{equation}
  H_{i-2}C=H_{i-1}C. \label{eq:equivalent_ops}
\end{equation}
Invoking property (3) again, either
\begin{equation}
  H_iH_{i-2}C=H_iH_{i-1}C=H_{i-1}^2C\quad\text{or}\quad H_iH_{i-2}C=H_{i+1}H_{i-2}C=C.\label{eq:Hi+1sqrd}
\end{equation}
Once more, the second case results in a contradiction; hence repeatedly
using (2),
\begin{equation*}
  H_{i+2}H_i^{2k+1}H_{i-2}C=H_i^{2k}H_{i+2}H_iH_{i-2}C=H_i^{2k}H_{i+2}H_{i-1}^2C=H_i^{2k}H_{i-1}C,
\end{equation*}
implying a contradiction in the case where $m$ is odd.
In the case that $m$ is even, repeatedly using (2), \eqref{eq:equivalent_ops}
and (5), we have
\begin{equation*}
  H_{i+2}H_i^{2k}H_{i-2}C = H_i^{2k}H_{i+2}H_{i-2}C=H_i^{2k}H_{i+2}H_{i-1}C
    =H_i^{2k}C,
\end{equation*}
obtaining another contradiction; it follows that it is impossible that
every shortest path can be written as a word containing no more than 2 of the
operators $H_i$.

We next show that if the two operators are $H_i$ and $H_{i+2}$ for some $i$, we
may rewrite the word in terms of $H_i$ and $H_{i-1}$, or $H_{i+1}$ and $H_i$.
In \eqref{eq:Hi+1sqrd}, we showed that for any cell $C$ contained in a shortest
path $H_iH_{i-2}C=H_{i-1}^2C$. Suppose that a shortest path is represented as a 
product of $H_{i-2}$ and $H_i$. Then each pair $H_iH_{i-2}$ may be replaced by
$H_{i-1}^2$, and using (2) to permute each of these pairs to the right,
we eventually obtain one of
\begin{equation*}
  H_i^mH_{i-1}^{2n}\quad\text{or}\quad H_{i-2}^mH_iH_{i-1}^{2n}.
\end{equation*}
In the first case, the proof is complete. In the second case, it must be that
$H_iH_{i-1}^{2n}=H_{i-1}^{2n+1}$ or else $m=0$, since otherwise (3) implies
\begin{equation*}
  H_{i-2}^mH_iH_{i-1}^{2n} = H_{i-2}^mH_{i+1}H_{i-1}^{2n} = H_{i-2}^{m-1}H_{i-1}^{2n},
\end{equation*}
which is a contradiction. Hence we have proved the
claim, and in fact since we have obtained a shortest path in the form
\eqref{eq:shortest_path_ops}, the lemma is proven for this particular case.

Finally, we consider a general word made up of only the operators $H_i$ and
$H_{i+1}$. Now consider a word of the form $H_{i+1}H_i^mH_{i+1}$. If $m$ is
even, then we can generate a new word corresponding to a shortest path
$H_{i+1}^2H_i^m$. If $m$ is odd, then we can write a new shortest path as
\begin{equation*}
  H_{i+1}H_iH_{i+1}H_i^{m-1}.
\end{equation*}
But then using (3), this must be able to be written either as
\begin{equation*}
  H_{i+1}^3H_i^{m-1},\quad\text{or}\quad H_{i+2}H_{i-1}H_{i+2}H_i^{m-1},
\end{equation*}
where the second case results in a contradiction. It is now possible to
check that this implies the full conclusion, since by these arguments we
can always transform a general word composed of $H_i$ and $H_{i+1}$ in
one of the form \eqref{eq:shortest_path_ops}.

\begin{figure}
  \includegraphics[height=6cm]{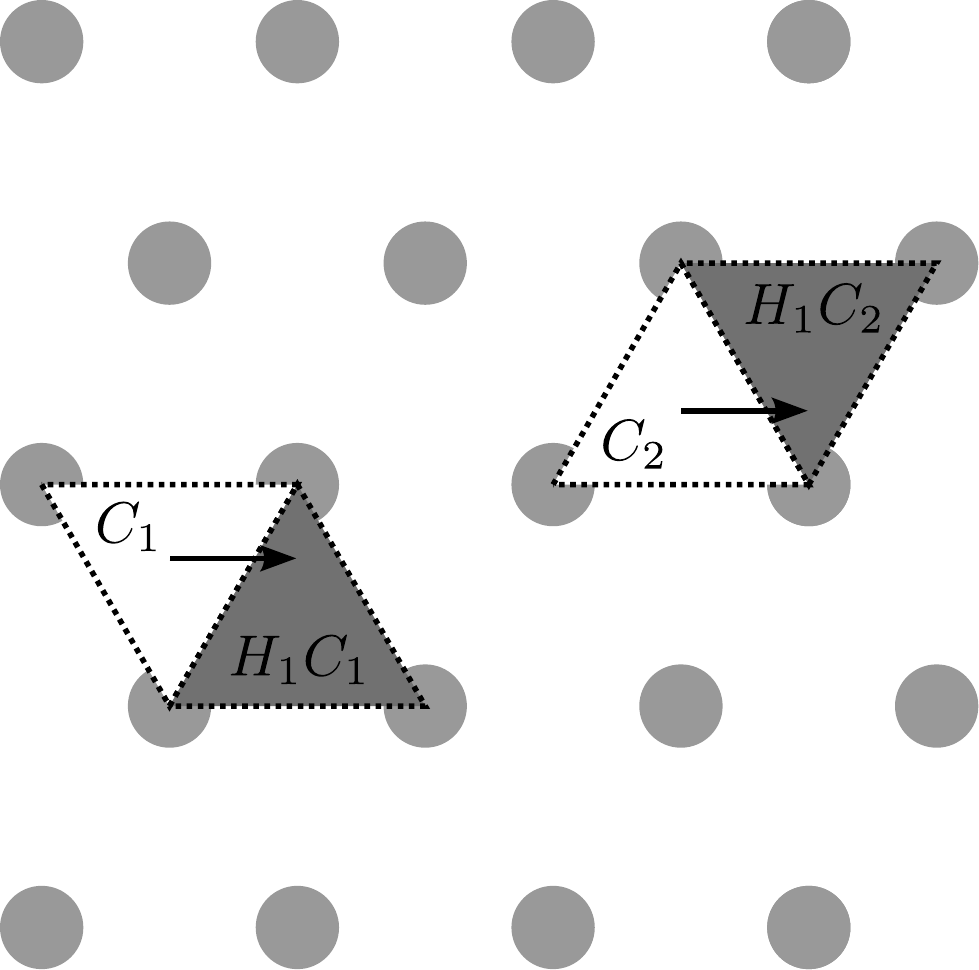}
  \caption{This illustration shows the operation of the operator $H_{a_1}$ on
  two representative cells, $C_1$ and $C_2$.}
  \label{fig:hopops}
\end{figure}

\subsection{Proof of Lemma \ref{lem:cut.lwr.bnd}}
To prove this lemma, we use Lemma \ref{th:straight_cuts} to assert that
$z$ should be made up of straight cuts. This will allow us to make estimates
for each straight segment, which depend upon on the orientation of each segment.
We divide the lattice into `sextants' by defining
\begin{equation*}
  \mathcal{S}^i:=\b\{x\in\R^2 \bsep \exists\,\lambda>0,\mu\geq0\text{ such that }
    x = \lambda a_i + \mu a_{i+1}\b\}.
\end{equation*}
and rings $\mathcal{R}^i$ by defining $\mathcal{R}^{-1}=\emptyset$, $\mathcal{R}^0:=C_0$,
and then
\begin{equation*}
  \mathcal{R}^i=\mathrm{clos}\,\bigcup\b\{C\in\Cells\setminus \b(\mathcal{R}^{i-1}\cup\mathcal{R}^{i-2}\b)\bsep \mathrm{clos}(C)\cap\mathrm{clos}(\mathcal{R}^{i-1})\neq\emptyset\b\}.
\end{equation*}
We will say a straight cut $z$ is:
\begin{enumerate}
  \item {\em tangential} if $\supp\{z\}\subseteq\mathcal{S}^i$ and the cut direction is
    $a_{i+2}$, or equivalently $\supp\{z\}\subseteq\mathcal{S}^i\cap\mathcal{R}^r$ for
    some $r$, and
  \item {\em radial} if $\supp\{z\}\subseteq \mathcal{S}^i$ and it has direction $a_i$ or
    $a_{i+1}$.
\end{enumerate}
It may be checked that for any straight cut, there exist $C^+,C^-\in\Cells$
such that
\begin{equation*}
  \int_{\partial C^\pm} z = \pm1.
\end{equation*}

We separate the full result into 2 further lemmas, which give precise estimates for each
of these classes of straight cuts, before combining them to complete the proof
for the general case. By Taylor expanding $\psi'(\alh_b)$ around $0$ we write
\begin{equation}
 \sum_{b\in\Bonds} \psi'(\alh_b)z_b = \sum_{b\in\Bonds} \psi''(0)\alh_b\,z_b
    +\smfrac16\psi^{(4)}(s_b)(\alh_b)^3z_b.\label{eq:tayl_exp_z_sum}
\end{equation}
We proceed to estimate the first terms in the summand, by estimating on the radial and
tangential straight segments of $z^m$ separately.

\begin{lemma}
\label{lem:tan_cut}
For a tangential cut $z^{\mathrm{tan}}$ with $\|z^\mathrm{tan}\|_1=\hop_2(C^+,C^-)=l$
on ring $\mathcal{R}^r$, we have the following estimate:
\begin{equation*}
  \sum_{b\in\Bonds} z^{\mathrm{tan}}_b \alh_b\geq
    -\frac{1}{2\pi}\arctan\bg(\frac{2\lfloor\min(l,2r-l)/2\rfloor+1}{(r-2/3)\sqrt{3}}\bg)
    -O(r^{-1}).
\end{equation*}
\end{lemma}

\begin{proof}
First, we appeal to symmetry. If $b=(\xi,\xi+a_j)\in\mathcal{S}_i$, then
applying the reflection
\begin{equation*}
  R=\frac13(a_i+a_{i+1})\otimes(a_i+a_{i+1})-a_{i+2}\otimes a_{i+2},
\end{equation*}
it is straightforward to check that if $b'=(R\xi,R\zeta)$, then
\begin{equation*}
  \alh_b=-\alh_{b'}.
\end{equation*}
This means that if a tangential cut crosses the line of symmetry
$\{t(a_i+a_{i+1})\sep t\in\R\}$, then some of the bond contributions cancel.
It follows that we need only consider the case where all bonds in the
support on $z^\mathrm{tan}$ lie on one side of this line of symmetry, since
this is the worst case. For such cuts, it may be checked that $l\leq r$.

Identifying $z^\mathrm{tan}$ with a dual lattice path, we may
enumerate $b_k\in\supp\{z^\mathrm{tan}\}$ `along the path', and elementary
geometry now shows that $\alh_{b_k}$ has alternating sign as $k$
increases. Letting $\xi^r:= \smfrac12(r-\smfrac23)(a_i+a_{i+1})$ if $i$ is odd,
$\xi^r:= \smfrac12(r-\smfrac13)(a_i+a_{i+1})$ if $i$ is even,
it is straightforward to show that each bond $b_k$ can be represented as one of
\begin{equation*}
  (\xi^r+sa_{i+2},\xi^r+sa_{i+2}+a_i)\quad\text{or}\quad(\xi^r+sa_{i+2},\xi^r+sa_{i+2}+a_{i+1})
\end{equation*}
where $s\in\b\{0,\smfrac{1}{2},\ldots,\smfrac{r-1}{2}\b\}$.
Elementary trigonometry now allows us to calculate that
\begin{align*}
  2\pi\alh_b\,z_b^\mathrm{tan} &= \pm\bg[\arctan\bg(\frac{s+1/2}{|\xi^{r+1}|}\bg)
    -\arctan\bg(\frac{s}{|\xi^r|}\bg)\bg]\\
  &\text{or}\,\pm\bg[\arctan\bg(\frac{s-1/2}{|\xi^{r+1}|}\bg)-\arctan\bg(\frac{s}{|\xi^r|}\bg)\bg]
\end{align*}
respectively for the cases above. We therefore have that
\begin{align*}
  2\pi\sum_b z^{\mathrm{tan}}_b\alh_b &= \sum_{t=1}^{\lfloor l/2\rfloor}
    \arctan\B(\smfrac{t+s_0+1/2}{|\xi^{r+1}|}\B)
    -2\arctan\B(\smfrac{t+s_0}{|\xi^r|}\B)
    +\arctan\B(\smfrac{t+s_0-1/2}{|\xi^{r+1}|}\B)+O(r^{-1}),\\
  \text{or }&=\sum_{t=1}^{\lfloor l/2\rfloor}
    \arctan\B(\smfrac{t+s_0+1/2}{|\xi^r|}\bg)
    -2\arctan\B(\smfrac{t+s_0}{|\xi^{r+1}|}\B)
    +\arctan\B(\smfrac{t+s_0-1/2}{|\xi^r|}\B)+O(r^{-1}),
\end{align*}
where $s_0\in\b\{0,\smfrac{1}{2},\ldots,\smfrac{r-1}{2}\b\}$, and 
the $O(r^{-1})$ term arises from the contribution of at most 2 bonds we have neglected, whose
contribution we estimate using \eqref{eq:alpha.hat.bnd}.
In the second case, the fact that $-\arctan$ is convex for positive arguments
implies that the sum is bounded below by $0$. In the first case, for all $t$ in the range of
summation,
\begin{equation*}
  \frac{t+s_0+1/2}{|\xi^{r+1}|}\geq\frac{t+s_0}{|\xi^r|}\quad\text{and trivially}\quad
    \frac{t+s_0}{|\xi^{r+1}|}\leq\frac{t+s_0}{|\xi^r|}.
\end{equation*}
Since $\arctan$ is increasing, we obtain the lower bound
\begin{align*}
  2\pi\sum_b z^{\mathrm{tan}}_b\alh_b&\geq \sum_{t=1}^{\lfloor l/2\rfloor}\bg\{\arctan\bg(\frac{t+s_0-1/2}{|\xi^r|}\bg)
    -\arctan\bg(\frac{t+s_0}{|\xi^r|}\bg)\bg\}+O(r^{-1})\\
  &\geq-\arctan\bg(\frac{\lfloor l/2\rfloor+1/2}{|\xi^r|}\bg)+O(r^{-1}),
\end{align*}
using the fact that arctan is positive and increasing for positive arguments.
Finally, note that in the case where $l>r$, i.e. the tangential cut crosses
the line of symmetry, and we obtain the same estimate but with $2r-l$ in place
of $l$ in the formula above, so using the definition of $\xi^r$ gives the result.
\end{proof}

\begin{lemma}
\label{lem:rad_cut}
For a radial cut $z^{\mathrm{rad}}$ such that $\|z^\mathrm{rad}\|_1=\hop_2(C^+,C^-)=l$,
either $|x^{C^+}|<|x^{C^-}|$ and
\begin{equation*}
  \sum_{b\in\Bonds} z_b^{\mathrm{rad}} \alh_b \geq 0,
\end{equation*}
or else $|x^{C^+}|>|x^{C^-}|$, and if $C^-\in\mathcal{R}^r$, then
\begin{equation*}
  \sum_{b\in\Bonds} z_b^{\mathrm{rad}} \alh_b\geq
    -\frac1\pi\arcsinh\bg(\frac{2\lceil l/2\rceil}{\sqrt3 (r-2/3)}\bg)\bg)-O(r^{-1}).
\end{equation*}
\end{lemma}

\begin{proof}
First, we enumerate the bonds in $b_k\in\{b\in\Bonds\sep z^\mathrm{rad}_b>0\}$,
beginning with the bond for which $\d{b}$ is smallest, and proceeding
outwards along the cut. Elementary geometry shows that the terms
$z_{b_k}^{\mathrm{rad}}\alh_{b_k}$ are all positive in the case where
$b_1$ is in one of the directions $a_{i+1}$, $a_{i+2}$ or $a_{i+3}$,
which corresponds to having $|x^{C^+}|<|x^{C^-}|$; this immediately 
provides the first bound.

In the second case \eqref{eq:alpha.hat.bnd} implies
\begin{equation*}
  \alh_{b_k}\geq -\frac{1}{2\pi\d{b_k}}.
\end{equation*}
Without loss of generality, we assume the cut direction is $a_i$, the case with
direction $a_{i+1}$ being similar.
There are now two cases: $b_1$ is either in the direction $a_{i-1}$, or 
$a_{i-2}$. Further elementary geometry allows us to conclude that in the first
case, $\d{b_1}=|x|$ for some $x\in\mathcal{S}^i$, and in the second, $\d{b_1}>|x|$
with $x=x^{C^-}\in\mathrm{clos}(\mathcal{S}^i)$. In either case, $\d{b_2}$
satisfies the same lower bound as $\d{b_1}$, and further, we have that
\begin{equation*}
  \d{b_{2n-1}},\d{b_{2n}}\geq|x+na_i|.
\end{equation*}
Noting that as $x\in\mathrm{clos}(\mathcal{S}^i)$, it follows that $a_i\cdot x\geq0$
and
\begin{equation*}
 \frac{1}{|x+ta_i|} = \frac{1}{\sqrt{|x|^2+2t\,a_i\cdot x +t^2}}\leq\frac{1}{\sqrt{|x|^2+t^2}},
\end{equation*}
which is a decreasing function of $t$, so we estimate
\begin{equation*}
  \sum_{k=1}^l\alh_{b_k}\geq -\frac 1\pi\sum_{i=0}^{\lceil l/2\rceil}\frac{1}
  {|x+na_i|} \geq -\frac1\pi \bg(\frac{1}{|x|}+\int_0^{\lceil l/2 \rceil}
  \frac{1}{\sqrt{|x|^2+s^2}}\ds\bg).
\end{equation*}
Evaluating the integral, and noting further that $|x|\geq \smfrac{\sqrt{3}}{2}(r-2/3)$,
we obtain the conclusion.
\end{proof}

We now combine the estimates of Lemma \ref{lem:tan_cut} and Lemma
\ref{lem:rad_cut} to obtain an estimate for a general cut $z^m$ made up of two
straight segments. As we showed in Lemma \ref{th:straight_cuts}, each $z^m$ is
made up of at most 2 straight segments. It may be checked that each of these
segments is either purely radial, purely tangential, or changes from tangential
to radial part way along its length, with one bond which crosses
$\partial \mathcal{S}^i$. All possible cuts satisfying the \DMCP~and
made up of 2 straight segments can therefore be decomposed as either
\begin{enumerate}
  \item a tangential cut and 2 radial cuts or
  \item a tangential cut, a radial cut and tangential cut,
\end{enumerate}
where any of these segments could possibly have length 0, and neglecting
the extra bonds mentioned above for now.
Recalling the result of Corollary \ref{th:str_hop_min}, 
\begin{equation*}
  \hop_2(C_0,C_m^-)\geq\hop_2(C_0,C_m^+).
\end{equation*}
Consider the first case, letting the two radial segments be of
length $l_1$ and $l_2$ respectively, and the tangential segment
of length $l-l_1-l_2$. If $|x^{C_m^+}|<|x^{C_m^-}|$
and the radial cuts are of non-zero length, then the radial segments
have the trivial lower bound, by Lemma \ref{lem:rad_cut}. In the worst
case, where $l-l_1-l_2=r$, we have the bound
\begin{equation*}
  \sum_{b\in\Bonds}z^m_b\,\alh_b\geq-\frac{1}{2\pi}\arctan\bg(
    \frac{r+1}{(r-2/3)\sqrt{3}}\bg)
    -O(r^{-1})=-\frac{1}{12}-O(r^{-1}).
\end{equation*}
Otherwise, $|x^{C_m^+}|>|x^{C_m^-}|$, so applying Corollary
\ref{th:str_hop_min}, we have that for $C'\in\mathcal{R}^r$,
\begin{equation}
  2r\geq \hop_2(C_0,C')\geq\hop_2(C_m^+,C') = l_1+l_2.
  \label{eq:rad_cuts_max_length}
\end{equation}
Therefore, applying Lemma \ref{lem:tan_cut} and Lemma \ref{lem:rad_cut},
\begin{multline*}
  \sum_{b\in\Bonds}z_b\,\alh_b\geq -\frac{1}{2\pi}\Bg(\arctan\bg(
    \frac{2\lfloor (l-l_1-l_2)/2\rfloor+1}{(r-2/3)\sqrt{3}}\bg)
    +2\,\arcsinh\bg(\frac{2\lceil l_1/2\rceil}{\sqrt3 (r-2/3)}\bg)\\
    +2\,\arcsinh\bg(\frac{2\lceil l_2/2\rceil}{\sqrt3 (r+\lfloor l_1/2\rfloor-2/3)}\bg)+O(r^{-1})\Bg).
\end{multline*}
By ignoring the floor functions, it is possible to check that under the bound
\eqref{eq:rad_cuts_max_length}, the function in parentheses is increasing in
both $l_1$ and $l_2$ if $r$ is suitably large;
we therefore have that the maximum must occur when $l=l_1+l_2$. Hence putting
$l_2=l-l_1$, we have
\begin{equation*}
  \sum_{b\in\Bonds}z_b\,\alh_b\geq -\frac{1}{\pi}\Bg( 
    \arcsinh\bg(\frac{2\lceil l_1/2\rceil}{\sqrt3 (r-2/3)}\bg)
    +\arcsinh\bg(\frac{2\lceil (l-l_1)/2\rceil}
      {\sqrt3 (r+\lfloor l_1/2\rfloor-2/3)}\bg)+O(r^{-1})\Bg).
\end{equation*}
Dropping the floor and ceiling functions, this estimate is convex in $l_1$.
The worst cases are therefore $l_1=0$ or $l_1=l$ and $2\lceil l/2\rceil =2r+1$,
giving the value
\begin{equation*}
  \sum_{b\in\Bonds} z_b\,\alh_b \geq -\frac1\pi\arcsinh\bg(
    \frac{2r+1}{\sqrt{3}(r-2/3)}\bg)+O(r^{-1})
    =-\frac{\arcsinh\b(2/\sqrt{3}\b)}{\pi}+O(r^{-1}).
\end{equation*}
In the case where we have tangential, radial and tangential segments, similar
arguments show that, once more, the worst possible bound arises in the case
where the cut is purely radial, giving the same lower bound.

We now return to the lower order terms in \eqref{eq:tayl_exp_z_sum}. Then
by crudely estimating
\begin{equation*}
  \bg|\sum_{b\in\Bonds} \smfrac16\psi^{(4)}(s_b)(\alh_b)^3z^m_b\bg|
    \lesssim \sum_{\d{b}\geq r} \d{b}^{-3} \lesssim \frac1r
\end{equation*}
and noting that the worst case bounds always have occur when
$\hop_2(C_0,C_m^-)\simeq r$, we have
\begin{equation*}
  \sum_{b\in\Bonds}z^m_b\psi^\prime(\alh_b)\geq -\psi^{\prime\prime}(0)
    \frac{\arcsinh\b(2/\sqrt{3}\b)}{\pi} - c_0\,\hop_2(C_0,C^-_m)^{-1}.
    % \label{eq:zm_psi_bnd}
\end{equation*}

\bibliographystyle{plain}
\bibliography{qc}
\end{document}